\documentclass[11pt,reqno]{amsart}
\usepackage[colorlinks=true,urlcolor=blue,linkcolor=red,citecolor=magenta]{hyperref}
\usepackage{eurosym}
\usepackage{amsfonts}
\usepackage{amsmath,amssymb,amsfonts}
\usepackage{amsfonts}
\usepackage{graphicx}
\usepackage{graphics}
\usepackage{url,hyperref}
\usepackage{xcolor}
\usepackage[margin=1.25in]{geometry}
\newtheorem{theorem}{Theorem}[section]

\newtheorem{lemma}[theorem]{Lemma}

\newtheorem{proposition}[theorem]{Proposition}

\newcommand{\R}{\mathbb{R}}
\newcommand{\G}{\mathcal{G}}

\begin{document}
%\linenumbers	
\title[Outer billiards in the space of oriented geodesics]{Outer billiards
in the spaces of oriented geodesics of the three-dimensional space forms}
\author{Yamile Godoy, Michael Harrison, and Marcos Salvai}
\thanks{This work was supported by Consejo Nacional de Investigaciones Cient%
\'{\i}ficas y T\'ecnicas and Secretar\'{\i}a de Ciencia y T\'ecnica de la
Universidad Nacional de C\'ordoba. We are very grateful to the referee for 
the careful reading of the paper and the detailed suggestions. 
The second author thanks the
hospitality of {\sc ciem-famaf} during his visit.}
%\date{ }
\maketitle

\begin{abstract}
Let $M_{\kappa }$ be the three-dimensional space form of constant curvature $%
\kappa =0,1,-1$, that is, Euclidean space $\mathbb{R}^{3}$, the sphere $S^{3}
$, or hyperbolic space $H^{3}$. Let $S$ be a smooth, closed, strictly convex
surface in $M_{\kappa }$. We define an outer billiard map $B$ on the four
dimensional space $\mathcal{G}_{\kappa }$ of oriented complete geodesics of $%
M_{\kappa }$, for which the billiard table is the subset of $\mathcal{G}%
_{\kappa }$ consisting of all oriented geodesics not intersecting $S$. We
show that $B$ is a diffeomorphism when $S$ is quadratically convex.

For $\kappa =1,-1$, $\mathcal{G}_{\kappa }$ has a K\"{a}hler structure
associated with the Killing form of $\operatorname{Iso}(M_{\kappa })$. We prove that 
$B$ is a symplectomorphism with respect to its fundamental form and that $B$
can be obtained as an analogue to the construction of Tabachnikov of the
outer billiard in $\mathbb{R}^{2n}$ defined in terms of the standard
symplectic structure. We show that $B$ does not preserve the fundamental
symplectic form on $\mathcal{G}_{\kappa }$ associated with the cross product
on $M_{\kappa }$, for $\kappa =0,1,-1$.

We initiate the dynamical study of this outer billiard in the hyperbolic case by introducing and
discussing a notion of holonomy for periodic points.
\end{abstract}


\medskip


\section{Introduction}

\subsection{Motivation\label{sec:motivation}}

The \emph{dual} or \emph{outer billiard map} $B$ is defined in the plane as
a counterpart to the usual inner billiards. Let $\gamma \subset \mathbb{R}%
^{2}$ be a smooth, closed, strictly convex curve, and let $p$ be a point
outside of $\gamma $. There are two tangent lines to $\gamma $ through $p$;
choose one of them consistently, say, the right one from the viewpoint of $p$%
, and define $B(p)$ as the reflection of $p$ in the point of tangency (see
Figure \ref{fig:outerplane}).

\begin{figure}[ht!]
\centerline{
\includegraphics[width=2in]{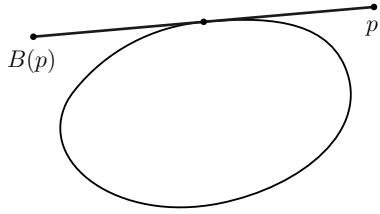}
}
\caption{The outer billiard map in the plane}
\label{fig:outerplane}
\end{figure}

The study of the dual billiard was originally popularized by Moser \cite
{moser1, moser2}, who considered the dual billiard map as a crude model for
planetary motion and showed that orbits of the map cannot escape to
infinity. The outer billiard map has since been studied in a number of
settings; see \cite{TabachMathIntell, TabachnikovOuter,
TabachnikovBilliards, bookT} for surveys.

In \cite{Tabachnikov}, Tabachnikov generalized planar outer billiards to
even-dimensional standard symplectic space $(\mathbb{R}^{2n},\omega )$ as
follows: given a smooth, closed hypersurface $M$ in $\mathbb{R}^{2n}$ which
is quadratically convex (that is, the shape operator at any point of $M$ is
definite), the restriction of $\omega $ to each tangent space $T_{q}M$ has a 
$1$-dimensional kernel, called the \emph{characteristic line}. If $\nu $ is
the outward-pointing unit normal vector field on $M$ and $\mathbb{R}^{2n}$
is identified with $\mathbb{C}^{n}$, then $\left\{ q-ti\nu (q)\mid t\in 
\mathbb{R}\right\} $ is the characteristic line at $q\in M$. Tabachnikov
showed that the collection of (geodesic) rays $\left\{ q-ti\nu (q)\mid
t>0\right\} $, indexed by $q\in M$, foliate the exterior $U$ of $M$; in
particular, for each $p\in U$ there exists a unique such ray passing through 
$p$. This gives a smooth outer billiard map taking $p$ to its reflection in
the corresponding tangency point: 
\begin{equation}
B:U\rightarrow U\text{,\ \ \ \ \ \ }q-ti\nu (q)\mapsto q+ti\nu (q)\text{.}
\label{BTabach}
\end{equation}

Tabachnikov proved that the map is a symplectomorphism of the exterior of $M$%
, and in \cite{TabachnikovPeriodic}, he showed that the number of $3$%
-periodic trajectories of the outer billiard map in $\mathbb{R}^{2n}$ is not
less than $2n$.

\subsection{Spaces of geodesics of space forms}

\label{sec:gintro} The goal of the present article is to define and study
outer billiards in another setting: on the space of oriented geodesics of
the three-dimensional space form $M_{\kappa }$ of constant curvature $%
\kappa= 0, 1, -1$, that is, Euclidean space $\mathbb{R}^{3}$, the sphere $%
S^3 $, or hyperbolic space $H^{3}$.

The \emph{space of oriented geodesics} $\mathcal{G}_{\kappa }$ is a
four-dimensional manifold whose elements are the oriented trajectories of
complete geodesics in $M_{\kappa }$. Elements of $\mathcal{G}_{\kappa }$ can
also be described as equivalence classes of unit speed geodesics, where $%
\gamma \sim \sigma $ if $\sigma \left( t\right) =\gamma \left(
t+t_{o}\right) $ for some $t_{o}\in \mathbb{R}$. When $\kappa=0, -1$, $%
\mathcal{G}_{\kappa }$ is the space of oriented lines in Euclidean or
hyperbolic space, which is diffeomorphic to $TS^{2}$, and $\mathcal{G}_{1}$
is the space of oriented great circles of $S^{3}$ (or equivalently, the
Grassmannian of oriented planes in $\mathbb{R}^4$), which is diffeomorphic
to $S^{2}\times S^{2}$ (see \cite{GW} or \cite{FM}). Historically, the space
of oriented geodesics is at the core of symplectic geometry through its
relationship with optics. This space possesses a rich geometry 
(for instance, it admits two natural K\"ahler structures), whose study 
began with Hitchin \cite{ Hitchin} and continued with \cite{gk1, salvaimm,
	Salvai2, GG, Alek, Anciaux}. 
It has been useful, for instance, in the
characterization of geodesic foliations \cite{salvaiolf, GShip,
GodoySalvaiCali, HarrisonMZ, HarrisonBLMS, HarrisonAGT, HarrisonT}.

\subsection{The definition of the outer billiard map on $\mathcal{G}_{%
\protect\kappa }$}

\label{sec:obdef}

Let $S$ be a smooth, closed, strictly convex surface in $M_{\kappa }$, that
is, for each $p\in S$, the complete totally geodesic surface tangent to $p$
intersects $S$ only at $p$, near $p$. By smooth we understand of class $\mathcal C^\infty$. 
We denote by $ \mathcal{M}$ the three-dimensional
space of geodesics which are tangent to~$S$, which  
can be naturally identified with $T^{1}S$, the unit tangent bundle of $S$ (we were inspired by \cite{brendan}).

We first consider the cases $\kappa =0,-1$. We define the \emph{billiard table} 
\begin{equation*}
\mathcal{U}=\mathcal{G}_{\kappa }-\left\{ \text{oriented geodesics
intersecting }S\right\} \text{;}
\end{equation*}%
this is an open submanifold of $\mathcal{G}_{\kappa }$ with boundary equal to $%
\mathcal{M}$. 

We say that two distinct
geodesics $\ell $ and $\ell ^{\prime }$ in $\mathcal{U}$ are in the \emph{%
outer billiard correspondence} if there exists a complete totally geodesic
surface $P$, tangent to $S$ at a point $p$, containing $\ell $ and $\ell
^{\prime }$, such that $\ell ^{\prime }$ can be obtained by parallel
translating $\ell $ along the shortest geodesic from $\ell $ to $p$, twice
the distance from $\ell $ to $p$; see Figure \ref{fig:outereuclidean}.

\begin{figure}[ht]
\centerline{
		\includegraphics[width=2.2in]
		{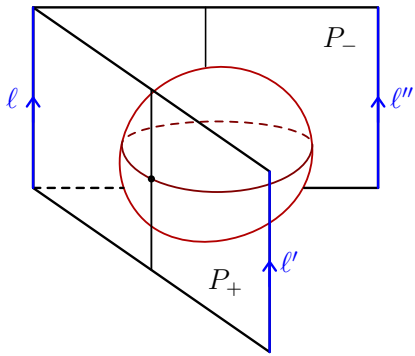}}
\caption{The geodesic $\ell$ is in the outer billiard correspondence with $%
\ell^{\prime }$ and $\ell^{\prime \prime }$}
\label{fig:outereuclidean}
\end{figure}

Given $\ell \in \mathcal{U}$, there exist exactly two complete totally
geodesic surfaces containing $\ell $ and tangent to $S$ 
(the assertion is clear for $\kappa=0$ and for $\kappa=-1$, for instance, 
using the Klein ball model of hyperbolic space, see Section \ref{Dynamics}). 
So $\ell $ is in
correspondence with exactly two other elements of $\mathcal{U}$. To define
the outer billiard map, we use the orientation of $\ell $ to choose the
surface on the right, say $P_{+}$, as follows.

Let $p_{\pm }$ be the point of tangency of the surface $S$ with $P_{\pm }$ and let $q_{\pm }$
be the point on $\ell $ realizing the distance $d_{\pm }$ to $p_{\pm }$ (see
Figure \ref{fig:outer}, left).

Let $\gamma _{\pm }$ be the geodesic ray joining $q_{\pm }$ with $p_{\pm }$,
with $\gamma _{\pm }(0)=q_{\pm }$ and $\gamma _{\pm }(d_{\pm })=p_{\pm }$.
Let $\alpha$ be a unit speed geodesic such that $\ell =[\alpha ]$ and define $t_{\pm }$ by $\alpha \left( t_{\pm
}\right) =q_{\pm }$ (see Figure \ref{fig:outer}, right). Let $W_{\pm }$ be
the parallel vector field along $\alpha $ with $W_{\pm }(t_{\pm })=\gamma
_{\pm }^{\prime }(0)$. Now choose the sign $+$ so that $\left\{
W_{+},W_{-},\alpha ^{\prime }\right\} $ is a positively oriented frame along 
$\alpha $. The \emph{outer billiard map} $B:\mathcal{U}\rightarrow \mathcal{U%
}$ can now be defined:
\begin{equation*}B(\ell ) \text{ is the oriented line obtained by parallel
translating $\ell $ along $\gamma _{+}$ between $0$ and $2d_{+}$}. 
\end{equation*}
Moreover, $B$ is a bijection.

\begin{figure}[h!t]
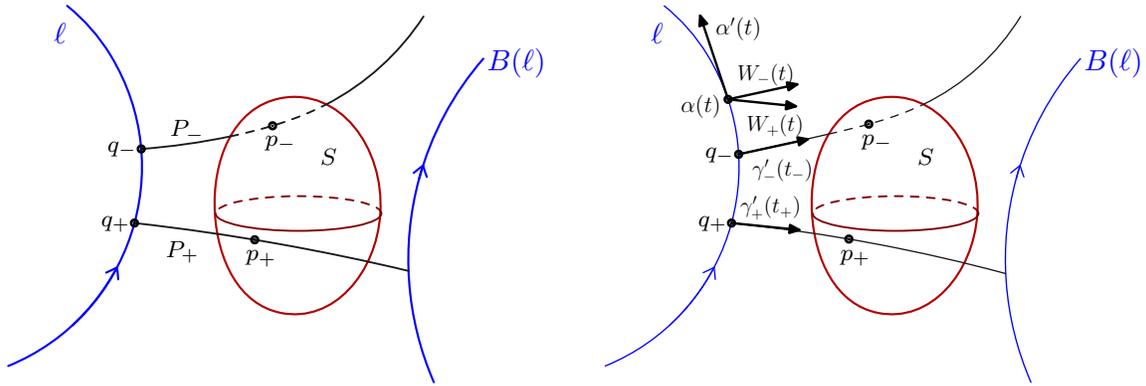

\centerline{
\includegraphics[height=2in]{newestouterhyp.mps} \hspace{.2in}
\includegraphics[height=2in]{newestouterhypframe.mps}
}
\caption{The outer billiard map on $\mathcal{L}_{\protect\kappa }$
associated with $S$}
\label{fig:outer}
\end{figure}

Notice that in the hyperbolic case (in contrast with the Euclidean),
parallel translating a line $\ell $ a distance $d$ along unit speed geodesic
rays $\gamma _{1}$ and $\gamma _{2}$ orthogonal to $\ell $ depends on the
initial points $\gamma _{1}(0)$ and $\gamma _{2}(0)$ in $\ell =\left[ \alpha %
\right] $, even if $\gamma _{2}^{\prime }(0)$ is the parallel transport of $%
\gamma _{1}^{\prime }(0)$ along $\alpha $.

Now we consider the case $\kappa =1$. We will see that a similar definition
of outer billiard map can be given. For an oriented great circle $c$ not intersecting $S$, 
there exist
exactly two great spheres of $S^{3}$ containing $c$ and tangent to $S$, but
there may be more than one (actually, a circle worth of them) shortest
geodesics between $c$ and the tangency point in $S$, so that $q_{+}$ or $%
q_{-} $ are not well defined. That is the case when the distance from $c$ to
the tangency point in $S$ is $\pi /2$. We call $\mathcal{C}$ the set of
these oriented great circles. We will describe this set later, in Subsection \ref{sphericalC}, in terms of
the Gauss map of $S$, $q\mapsto T_{q}S$, using the canonical identification
of oriented great circles with oriented planes through the origin in $%
\mathbb{R}^{4}$. Although the outer billiard map $B$ will be still
well-defined on $\mathcal{C}$ as the involution $c\mapsto -c$ (the same
circle with opposite orientation), it is easier to exclude $\mathcal{C}$
from the domain of definition. Thus, in the spherical case we define%
\begin{equation}
\mathcal{U}=\left\{ c\in \mathcal{G}_{1}\mid c\text{ does not intersect }%
S\right\} -\mathcal{C}\text{.}  \label{UmenosC}
\end{equation}

\begin{proposition}
	\label{esfera}Let $S$ be a strictly convex closed surface in $S^{3}$. The
	analogue of the outer billiard map in the cases $\kappa =0,-1$ is well defined on $%
	\mathcal{U}$ for $\kappa = 1$ and is a bijection onto this set.
\end{proposition}

For $\kappa =0,1,-1$, we call $B:\mathcal{U}\rightarrow \mathcal{U}$ as
above the \emph{outer billiard map on $\mathcal{G}_{\kappa }$ associated
with $S$.} We will show that $B$ is a diffeomorphism under the stronger
condition that $S$ is quadratically convex (in particular, strictly convex).

\begin{theorem}
\label{thm:diffeo}Let $S$ be a smooth, closed, quadratically convex surface
in the space form~$M_{\kappa }$. The outer billiard map $B:\mathcal{U}%
\rightarrow \mathcal{U}$ associated with $S$ is a diffeomorphism.
\end{theorem}

The following proposition shows that strict convexity is not enough for the smoothness of the billiard map. We present the example just for the sake of completeness, since it is  essentially the known corresponding fact for plane outer billiards.

\begin{proposition}\label{notSmooth}
	Let $S$ be a closed strictly convex surface in $\mathbb{R}^{3}$ which is
	invariant by the reflection with respect to the plane $y=0$ and contains the
	graph of the function $\varphi :\left( -2,2\right) \times \left( -1,1\right)
	\rightarrow \mathbb{R}$ defined by $\varphi \left( x,y\right) =f\left(
	x\right) +y^{2}$, where $f$ is a smooth function satisfying $f\left( 0\right) =f^{\prime }\left( 0\right)
	=f^{\prime \prime }\left( 0\right) =0$. Then the associated outer billiard
	map $B$ on $\mathcal{G}_{0}$ is not smooth.
\end{proposition}

\subsection{K\"{a}hler structures and the analogue of Tabachnikov's construction 
}

The space of oriented geodesics $\mathcal{G}_{\kappa }$ has one or two
canonical K\"{a}hler structures (for $\kappa =0$ or $\kappa =1,-1$,
respectively), so the natural question arises, whether Tabachnikov's
construction (\ref{BTabach}) of the outer billiard map for $\mathbb{R}%
^{2n} $ can be mimicked.

In order to deal with this issue, next we briefly introduce K\"{a}hler
structures on $\mathcal{G}_{\kappa }$, postponing formal definitions until
Section \ref{sec:prelim}.

Given $\ell \in \mathcal{G}_{\kappa }$, the $\pi /2$-rotation in $M_{\kappa
} $ which fixes $\ell $ induces a map on $\mathcal{G}_{\kappa }$ whose
differential at $\ell $ is a linear operator $\mathcal{J}_{\ell }$ on $%
T_{\ell }\mathcal{G}_{\kappa }$ which squares to $-$id. It may be
visualized by its action on four geodesic variations of $\ell $. The
geodesic $\ell $ may be translated in the two directions orthogonal to $\ell 
$ or rotated in two planes containing $\ell $, and the operator $\mathcal{J}%
_{\ell }$ sends translations to translations and rotations to rotations; see
Figure \ref{fig:J}. The collection of these linear transformations $\mathcal{%
J}_{\ell }$ is a complex structure $\mathcal{J}$ on $\mathcal{G}_{\kappa}$.

\begin{figure}[h!t]
\centerline{
		\includegraphics[width=4in]
		{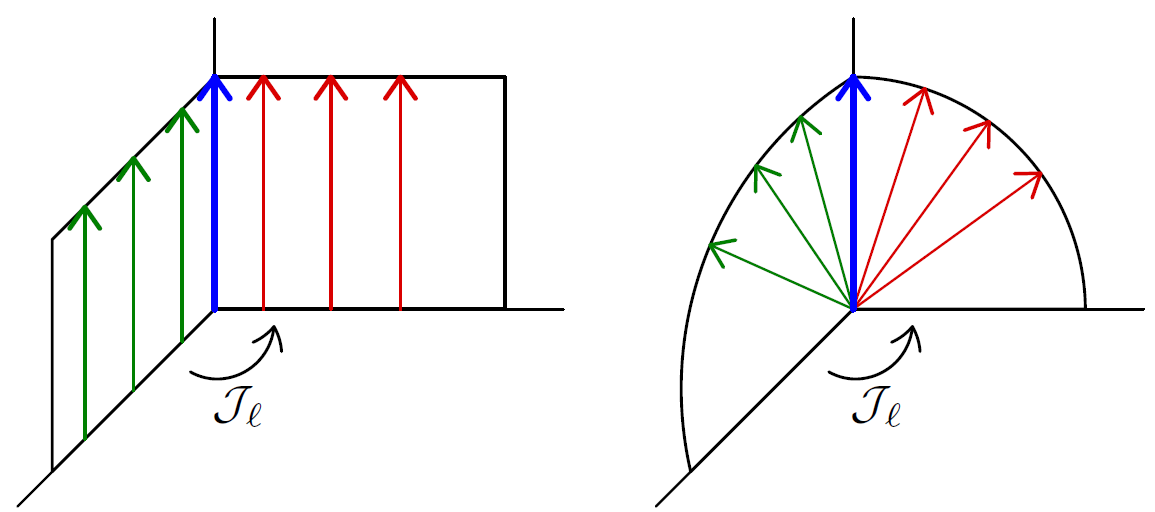}}
\caption{The linear transformation $\mathcal{J}_\ell$ maps the green
variation of geodesics to the red variation of geodesics}
\label{fig:J}
\end{figure}

For $\kappa =0,1,-1$, the manifold $\mathcal{G}_{\kappa }$ has a
pseudo-Riemannian metric $g_{\times }$ induced by the cross product on $%
M_{\kappa }$, and $(g_{\times },\mathcal{J})$ is a K\"{a}hler structure on $%
\mathcal{G}_{\kappa }$. For $\kappa = 1,-1$, there is an additional K\"{a}%
hler structure $(g_{K},\mathcal{J})$ on $\mathcal{G}_{\kappa }$, where $%
g_{K} $ is induced by the Killing form on Iso$(M_{\kappa })$. In the
Euclidean case, the Killing form $g_{K}$ degenerates. For the formal
definitions of $\mathcal{J}$, $g_{K}$ and $g_{\times }$ we use the language
of Jacobi fields; see Section \ref{sec:prelim}.

Having presented the K\"{a}hler structures on $\mathcal{G}_{\kappa }$, 
we can give a positive answer to the question in the beginning of the
subsection for $\kappa = 1,-1 $, using the K\"{a}hler structure $(g_{K},%
\mathcal{J})$. In this formulation we see that $B$ is a direct analogue of
the outer billiard map on $(\mathbb{R}^{2n},\omega )$ described in Section %
\ref{sec:motivation}.

Suppose that $S$ is a smooth, closed, strictly convex surface in $M_{\kappa
} $, and let $\mathcal{U}$ and $\mathcal{M}$ be as above. For $\ell \in 
\mathcal{M}$ let $\nu (\ell )$ be the outward-pointing unit normal
vector to $\mathcal{M}$ at $\ell $ (outward-pointing means pointing to $\mathcal{U}$). 
Given $\xi \in T_{\ell }\mathcal{M}$,
let $\Gamma _{\xi }$ denote the geodesic in $(\mathcal{G}_{\kappa },g_{K})$
with initial velocity $\xi $. Each geodesic $t\mapsto \Gamma _{\mathcal{J}%
\nu (\ell )}(t)$ traces out a totally geodesic surface in $M_{\kappa }$ that
is tangent to $S$. We will show that the collection of geodesic rays $%
\left\{ \Gamma _{\mathcal{J}\nu (\ell )}(t)\mid t<0\right\} $ indexed by $%
\ell \in \mathcal{M}$ foliates the exterior $\mathcal{U}$ of $\mathcal{M}$.
In particular, for each geodesic $\ell ^{\prime }\in \mathcal{U}$, there
exists a unique such ray passing through $\ell ^{\prime }$. This induces an
outer billiard map $B^{\prime }:\mathcal{U}\rightarrow \mathcal{U}$ taking $%
\ell ^{\prime }$ to its \textquotedblleft reflection\textquotedblright\ in
the tangent geodesic $\ell $: 
\begin{equation*}
B^{\prime }(\Gamma _{\mathcal{J}\nu (\ell )}(-t))=\Gamma _{\mathcal{J}\nu
(\ell )}(t)
\end{equation*}%
for $t>0$. In Section \ref{sec:kahlerdef}, after proving that $B^{\prime }$
is well defined, we will show that the outer billiard maps $B$ and $%
B^{\prime }$ coincide.

\begin{theorem}
\label{thm:equiv} For $\kappa = 1,-1$, let $S$ be a smooth, closed,
quadratically convex surface in $M_{\kappa }$, and consider $\mathcal{G}%
_{\kappa }$ endowed with the K\"{a}hler structure $(g_{K},\mathcal{J})$. The
map $B^{\prime }:\mathcal{U}\rightarrow \mathcal{U}$ coincides with the
outer billiard map $B$ on $\mathcal{U}$ associated with $S$.
\end{theorem}

The following proposition reveals that the K\"{a}hler structure $\left(
g_{\times },\mathcal{J}\right) $ is not appropriate in our setting, since it
does not give rise to a billiard map as in Tabachnikov's construction.

\begin{proposition}\label{gCross}
	For $\kappa =0,1,-1$, consider on $\mathcal{G}_{\kappa }$ the K%
	\"{a}hler structure $\left( g_{\times },\mathcal{J}\right) $. Let $S$ be as
	in the preceding theorem and let $\ell \in \mathcal{M}$. Then the metric $%
	g_{\times }$ degenerates on $T_{\ell }\mathcal{M}$, and for each vector $N$
	normal to $T_{\ell }\mathcal{M}$, the image of the geodesic in $\mathcal{G}%
	_{\kappa }$ with initial velocity $\mathcal{J}\left( N\right) $ is disjoint
	from $\mathcal{U}$.
\end{proposition}

\subsection{The symplectic
properties of the outer billiard map}

The outer billiard map interacts with the symplectic structures on $\mathcal{%
G}_\kappa$ as follows.

\begin{theorem}
\label{thm:symp} Let $B$ be the outer billiard map associated with a smooth,
closed, quadratically convex surface in $M_\kappa$.

\smallskip

a) For $\kappa =1,-1$, $B$ is a symplectomorphism with respect to the
fundamental symplectic form $\omega _{K}$ of $(\mathcal{G}_{\kappa },g_{K},%
\mathcal{J})$.

\smallskip

b) For $\kappa =0,1,-1$, $B$ does not preserve the fundamental symplectic
form $\omega _{\times }$ of $(\mathcal{G}_{\kappa },g_{\times },\mathcal{J})$%
.
\end{theorem}

Additional context for Theorem \ref{thm:symp} is given in Proposition \ref%
{prop:kx}, which establishes a relationship with plane hyperbolic outer
billiards (see \cite{planeHOB}) and supports the fact that $\omega _{K}$ (in
contrast with $\omega _{\times }$) is the natural symplectic form in our
context.

In the Euclidean case, the outer billiard map $B$ preserves parallelism,
yielding an $S^{2}$-worth of planar outer billiards as in Figure \ref%
{fig:outerplane}. That is, given a fixed direction $v$ in the  two-sphere, the
orthogonal projection of $S$ onto any plane $P$ orthogonal to $v$ determines
a smooth closed convex curve $\gamma $ in $P$ (the \emph{shadow} of $S$ with
respect to $v$). These shadows vary smoothly with respect to $v$, and the
outer billiard map $B$, restricted to lines with direction $v$, is
equivalent to the planar outer billiard in $P$ with respect to $\gamma $. In
particular, each such restriction is area-preserving.

Recall that for $%
\kappa =0$ the Killing form $g_{K}$ degenerates and so it does not induce a symplectic form on $\mathcal{G}_0$.
However, we have a
weaker structure, a Poisson bivector field $\mathcal{P}$, also compatible with $\mathcal{J}$, which we
define in Section \ref{sec:kahler}. The proof of the next proposition is immediate from the preceding paragraph. 

%The next proposition, which involves a Poisson bivector field $\mathcal{P}$
%on $\mathcal{G}_{0}$, is an immediate consequence of this (recall that for $%
%\kappa =0$ the Killing form $g_{K}$ degenerates; however, we have this
%weaker structure, Poisson, also compatible with $\mathcal{J}$, which we
%define in Section \ref{sec:kahler}).

\begin{proposition}
\label{Poisson}The outer billiard map associated with a smooth, closed,
quadratically convex surface in $\mathbb{R}^{3}$ preserves both the
canonical Poisson structure $\mathcal{P}$ on $\mathcal{G}_{0}$ and its
symplectic leaves, which are the submanifolds of parallel lines. In
particular, the restriction to each such submanifold is a symplectomorphism.
\end{proposition}

\subsection{Dynamical properties of the outer billiard $B$}

To begin the study of dynamical properties of the outer billiard map $B$, we
first observe that in the Euclidean case parallelism has consequences for
periodic orbits of the outer billiard map $B$ associated with the surface $S$%
. Given $v\in S^{2}$, consider the planar outer billiard system in any plane 
$P$ orthogonal to $v$, played outside the shadow of $S$ with respect to $v$.
For this planar system, there exist at least two distinct $n$-periodic
trajectories with rotation number $r$, for every $n\geq 3$ and positive $%
r\leq \left[ (n-1)/2\right] $ coprime with $n$ (see \cite[Theorem 6.2]{bookT}). Each such periodic orbit lifts to a
periodic orbit of the outer billiard map $B$ associated with $S$. In
particular, for each direction $v\in S^{2}$, there exist $n$-periodic orbits
consisting of geodesics with direction $v$.

The hyperbolic case is more interesting than the Euclidean one, because
unlike in $\mathbb{R}^{3}$, the lines $\ell $, $B(\ell )$, $B^{2}\left( \ell
\right) ,\dots $ in $H^{3}$ are (in general) not parallel, in the sense that
they are not orthogonal to a fixed totally geodesic surface. Indeed, if a
line $\ell $ in $H^{3}$ is parallel transported along a line $\ell ^{\prime }$
orthogonal to $\ell $, then $\ell ^{\prime }$ is the unique line preserved
by the one-parameter group of transvections along $\ell ^{\prime }$ 
(that is, isometries of $H^3$ translating $\ell ^{\prime }$ whose differentials realize the parallel transport along it). This nonparallel phenomenon is illustrated more explicitly
in the following proposition.

\begin{proposition}
\label{NotParallel} Given $\theta \in (0,\pi/2)$, there exist a
quadratically convex closed surface $S$ in $H^3$ and an oriented line $\ell$
not intersecting $S$ such that $\ell$ and $B^{3}(\ell)$ intersect at a point
forming the angle $\theta$.
\end{proposition}

Similarly, one can show the existence of a surface $S$ and $\ell $ such that 
$B^{3}(\ell )$ is different from $\ell $ and asymptotic to it.

We next introduce a notion of holonomy for periodic orbits of the outer
billiard map in hyperbolic space. Let $\bar{\ell}=(\ell _{0},\dots ,\ell
_{n-1})$ be an $n$-periodic orbit of $B$; in particular, we assume the $\ell
_{i}$ are distinct. For $0\leq k<n$ let $d_{k}$ be the signed distance
between the points $q_{+}(\ell _{k})$ and $q_{-}(\ell _{k})$, which were
defined in Section \ref{sec:obdef}, that is, if $\ell _{k}=\left[ \gamma _{k}%
\right] $ with $\gamma _{k}(0)=q_{-}$, then $\gamma _{k}(d_{k})=q_{+}$. Then
the \emph{holonomy of the periodic orbit} $\bar{\ell}$ is the number $%
d=\sum_{k=0}^{n-1}d_{k}$.

The definition also makes sense in Euclidean space, but every periodic orbit
has zero holonomy. We exhibit a periodic orbit in hyperbolic space with
nonzero holonomy.

\begin{proposition}
\label{holonomy}There exist a quadratically convex closed surface $S$ in $%
H^{3}$ and an oriented line $\ell$ not intersecting $S$ which is a periodic
point of the associated outer billiard map and whose holonomy is not zero.
\end{proposition}

We comment on the choice of the word holonomy in this context. For $\kappa
\leq 0$, let $\varpi :P=T^{1}M_{\kappa }\rightarrow \mathcal{G}_{\kappa }$
be the tautological line bundle, that is $\varpi (v)=[\gamma _{v}]$. It is
an $(\mathbb{R},+)$-principal bundle. The right action $\rho :P\times 
\mathbb{R}\rightarrow P$ is given by $\rho (u,t)=\gamma _{u}^{\prime }(t)$.
The line bundle $P_{\mathcal{U}}=\varpi ^{-1}(\mathcal{U})\rightarrow 
\mathcal{U}$ has two distinguished sections: $\sigma _{\pm }([\gamma
])=\gamma ^{\prime }(0)$ if $\gamma (0)=q_{\pm }$. The holonomy makes sense
only for periodic points and it is not associated with a particular
connection, but rather with a combination of the Levi-Civita connection on $%
M_{\kappa }$ and the flat connections induced by $\sigma _{\pm }$ along the
shortest segments joining $\ell _{k}$ with $\ell _{k+1}$.

While these results only provide the first steps towards understanding the
dynamical properties of the outer billiard map $B$ in hyperbolic space, they
also hint at the complexity and richness of the billiard system. Most of the
natural dynamical questions for the hyperbolic outer billiard map, for
example, regarding the existence of periodic orbits, remain open.

\section{Preliminaries}

\label{sec:prelim}

\subsection{The outer billiard map in the spherical case}\label{sphericalC}

We have presented in Subsection \ref{sec:obdef} the outer billiard map
associated with $S$ for $\kappa =0,1,-1$. The construction is clear for $\kappa =0,-1$.
Now we return to the spherical case. Before proving Proposition \ref{esfera}%
, we comment on the set $\mathcal{C}$ which we cut out of the billiard
table; see (\ref{UmenosC}). 
We define the map $\psi :S\rightarrow \mathcal{%
	C}$ as follows: Given $p\in S$, let $\psi \left( p\right) $ be the oriented
great circle obtained by intersecting $S^{3}$ with the subspace $%
T_{p}S\subset \mathbb{R}^{4}$ (we identify $T_p\mathbb{R}^{4}$ with $\mathbb{R}^{4}$ in the usual way), endowed with the orientation induced by that
of $T_{p}S$. Equivalently, and without using the immersion of the sphere in $%
\mathbb{R}^{4}$, for any positively oriented orthonormal basis $\left\{
u,v\right\} $ of $T_{p}S$, $\psi \left( p\right) =\left[ C_{p}\right] $,
where 
\begin{equation*}
	C_{p}\left( s\right) =\text{Exp}_{p}\left( \tfrac{\pi }{2}\left( \cos
	s~u+\sin s~v\right) \right) \cong \cos s~u+\sin s~v\text{.}
\end{equation*}%
With this notation, $\mathcal{C}$ is the image of $\psi $. Notice that if $c$
is an oriented circle in $\mathcal{C}$, then $-c$ is also in $\mathcal{C}$.

%\medskip

\begin{proof}[Proof of Proposition \protect\ref{esfera}]
	We verify that the same procedure as for the Euclidean and hyperbolic spaces
	applies here. By \cite{doCarmoWarner}, $S$ is contained in a hemisphere, say, the northern hemisphere $S_+^3$. 
	
	Let $\Pi$ denote the central projection from $S_+^3$ to the tangent plane $\R^3$ at the north pole (the so-called Beltrami map). It preserves strict convexity, since half great spheres in $S_+^3$ are mapped to affine $2$-planes and the order of contact is maintained by diffeomorphisms.
	
	Consider a great circle $c \in  \G_1 - \left\{ \mbox{oriented great circles intersecting } S \right\}$.  If $c$ is not contained in the equator $S^2 = \partial S_+^3$, then $\Pi(c)$ is an oriented affine line $\ell$ in $\R^3$.  Now, as shown in Section \ref{sec:obdef} for $\kappa=0$, there exist exactly two affine planes $\bar{P}_\pm$ containing $\ell$ and tangent to $\Pi(S)$, and $\Pi^{-1}(\bar{P}_\pm)$ are the desired great spheres containing $c$ and tangent to $S$ at points $p_\pm$. 
	Notice that $B(c) \in \mathcal U$, that is,  $B(c)$ is disjoint from $S$, since this holds for the lines in $\mathbb R^3$ corresponding to $c$ and $B(c)$.
	
	Now suppose that $c$ is contained entirely in the equator $S^2$.   Since $S$ is a positive distance from the equator $S^2$, any sufficiently small perturbation of the latter does not intersect $S$.  In particular, we may perturb the equator to a great sphere which does not contain the circle $c$ and argue as in the above paragraph.
	
	Let $d$ be the distance from $p_+$ to $c$.  If $d < \pi/2$, there exists a unique point $q_+ \in c$ realizing the distance, and the outer billiard map is well-defined.   If $d = \pi/2$, then $c$ belongs to $\mathcal{C}$, and so it is not in $\mathcal{U}$. The construction continues as in the Euclidean and hyperbolic cases.
	\end{proof}

We observe that if $q$ is any point of $c\in \mathcal{C}$, then the parallel
transport of $c$ along the geodesic joining $q$ with $p_{+}$ between $0$ and 
$\pi $ is a rotation by $\pi $. So the image of $c$ is the same great circle
with the opposite orientation (it would hold $B(c)=-c$, had we not excluded $%
\mathcal{C}$ from the billiard table).

\subsection{The Jacobi fields of the three dimensional space forms}

\label{sec:jacobi}

Here we provide a brief review of Jacobi fields, which arise naturally when
studying variations of geodesics, and thus play a central role in the proofs
of the main theorems. A more thorough treatment can be found in any standard
Riemannian geometry text, for example \cite{docarmo}.

Let $M$ be a complete Riemannian manifold and let $\gamma $ be a unit speed
geodesic of $M$. A Jacobi field $J$ along $\gamma $ is by definition a
vector field along $\gamma $ arising via a variation of geodesics as
follows: Let $\delta >0$ and $\phi :\mathbb{R}\times (-\delta ,\delta
)\rightarrow M$ be a smooth map such that $r\mapsto \phi (r,s)$ is a
geodesic for each $s\in (-\delta ,\delta )$ and such that $\phi (r,0)=\gamma
(r)$ for all $r$. Then 
\begin{equation*}
J(r)=\left. \tfrac{d}{ds}\right\vert _{0}\phi \left( r,s\right) \text{.}
\end{equation*}

Let $(M_{\kappa },\left\langle ,\right\rangle _{\kappa })$ denote the
three-dimensional complete simply connected manifold of constant sectional
curvature $\kappa $. The curvature tensor of $M_{\kappa }$ is given by 
\begin{equation}
R_{\kappa }(x,y)z=\kappa \left( \left\langle z,x\right\rangle _{\kappa
}y-\left\langle z,y\right\rangle _{\kappa }x\right) \text{,}
\label{curvatureK}
\end{equation}%
and the Jacobi fields along a geodesic $\gamma $ and orthogonal to $\gamma
^{\prime }$ are exactly the vector fields $J$ along $\gamma $ satisfying $%
\left\langle J,\gamma ^{\prime }\right\rangle _{\kappa }=0$ and 
\begin{equation}
\tfrac{D^{2}J}{dr^{2}}+\kappa J=0\text{,}  \label{equationJNormal}
\end{equation}%
where $\frac{D}{dr}$ is the covariant derivative associated with the Levi Civita connection of $M_\kappa$. 
Following a common abuse of notation, given a smooth vector field $J$ along
a curve $\gamma $, we write $J^{\prime }=\frac{DJ}{dr}$ if there is no
danger of confusion.

A Jacobi field $J$ along $\gamma $ is determined by the values $J(0)$ and $%
J^{\prime }(0)$ in the following way. Suppose that a Jacobi field $J$ along $%
\gamma $ satisfies $J(0)=u+a\gamma ^{\prime }(0)$ and $J^{\prime
}(0)=v+b\gamma ^{\prime }(0)$ where $a,b\in \mathbb{R}$ and $u,v\in \gamma
^{\prime \bot }$. Let $U$ and $V$ be the parallel vector fields along $%
\gamma $ with $U(0)=u$ and $V(0)=v$. Then 
\begin{equation}
J(r)=c_{\kappa }(r)U(r)+s_{\kappa }(r)V(r)+(a+rb)\gamma ^{\prime }(r)\text{,}
\label{JacobiK}
\end{equation}%
where 
\begin{equation*}
\begin{array}{ccc}
c_{1}(r)=\cos r\text{,} & c_{0}(r)=1\text{,} & c_{-1}(r)=\cosh r\text{,} \\ 
s_{1}(r)=\sin r\text{,} & s_{0}(r)=r\text{,} & s_{-1}(r)=\sinh r\text{.}%
\end{array}%
\end{equation*}%
Note that $s_{\kappa }^{\prime }=c_{\kappa }$ and $c_{\kappa }^{\prime
}=-\kappa s_{\kappa }$. Equation (\ref{JacobiK}) will allow us to perform
most computations without having to resort to coordinates of $M_{\kappa }$
or a particular model of it.

Next we see that the tangent vectors to the space $\mathcal{G}_{\kappa }$ at
an oriented geodesic $[\gamma ]$ may be identified with Jacobi fields along $%
\gamma $. Let $\gamma $ be a complete unit speed geodesic of $M_{\kappa }$
and let $\mathfrak{J}_{\gamma }$ be the space of all Jacobi fields along $%
\gamma $ which are orthogonal to $\gamma ^{\prime }$. There is a canonical
isomorphism 
\begin{equation}
T_{\gamma }\colon \mathfrak{J}_{\gamma }\rightarrow T_{[\gamma ]}\mathcal{%
\mathcal{G}_{\kappa }}\text{,}\hspace{1cm}T_{\gamma }(J)=\left. {\tfrac{d}{ds%
}}\right\vert _{0}[\gamma _{s}]\text{,}  \label{isoT}
\end{equation}%
where $\gamma _{s}$ is any variation of $\gamma $ by unit speed geodesics
associated with $J$. Moreover, if $J$ is the Jacobi field associated with a
variation $\phi :\mathbb{R}\times \left( -\delta ,\delta \right) \rightarrow
M_{\kappa }$ of $\gamma $ by unit speed geodesics ($J$ is not necessarily
orthogonal to $\gamma ^{\prime }$), then 
\begin{equation}
T_{\gamma }(J^{N})=\left. \tfrac{d}{ds}\right\vert _{0}[\phi _{s}],
\label{projectionK}
\end{equation}%
where $J^{N}(r)=J(r)-\langle J(r),\gamma ^{\prime }(r)\rangle _{\kappa
}\gamma ^{\prime }(r)$ (see Section 2 in \cite{Hitchin} or \cite{Salvai2}).

We offer one simple but useful application of the isomorphism $T_{\gamma }$.
Let $S$ be a smooth, closed, strictly convex surface in $M_{\kappa }$, and
let $\mathcal{M}\subset \mathcal{G}_{\kappa }$ be the collection of oriented
geodesics which are tangent to $S$.

\begin{lemma}
\label{tangentMscript} The isomorphism $T_\gamma$ identifies $\left\{ K\in 
\mathfrak{J}_{\gamma }\mid K( 0) \in T_{\gamma \left( 0\right) }S\right\} $
with the tangent space $T_{\left[ \gamma \right] }\mathcal{M}$.
\end{lemma}

\begin{proof}
It suffices to show that $T_{[\gamma ]}\mathcal{M}\subset T_{\gamma }
(\left\{ K\in \mathfrak{J}_{\gamma }\mid K\left( 0\right) \in
T_{\gamma\left( 0\right) }S\right\} )$, since both spaces have dimension $3$.
Let $X\in T_{[\gamma ]}\mathcal{M}$ and let $c$ be a smooth curve on $%
\mathcal{M}$ (defined on an interval $I$ containing~$0$) such that $%
c(0)=[\gamma ]$ and $c^{\prime }(0)=X$. For each $s\in I$, let $[\gamma
_{s}]\in \mathcal{M}$ such that $\gamma _{s}(0)\in S$ and $[\gamma
_{s}]=c(s) $. By (\ref{projectionK}), $X=T_{\gamma }(J^{N})$, where $J$ is
given by 
\begin{equation*}
J(r)=\left. \tfrac{d}{ds}\right\vert _{0}\gamma _{s}(r).
\end{equation*}%
Now $\gamma ^{\prime }(0)\in T_{\gamma (0)}S$, and since $s\mapsto \gamma
_{s}(0)$ is a smooth curve on $S$, $J(0)=\left. \frac{d}{ds}\right\vert
_{0}\gamma _{s}(0)\in T_{\gamma (0)}S$. Therefore, $J^{N}\left( 0\right)
=J(0)-\langle J(0),\gamma ^{\prime }(0)\rangle _{\kappa }\gamma ^{\prime
}(0)\in T_{\gamma (0)}S$, as desired.
\end{proof}

\subsection{K\"{a}hler structures on the spaces of oriented geodesics}

\label{sec:kahler}

In Section \ref{sec:gintro} we introduced the two canonical K\"{a}hler
structures $\left( g_{K},\mathcal{J}\right) $ and $\left( g_{\times },%
\mathcal{J}\right) $ on the space of oriented geodesics $\mathcal{G}_{\kappa
}$, for $\kappa =1,-1$, and also the K\"{a}hler structure $\left( g_{\times },%
\mathcal{J}\right) $ and the Poisson bivector field $\mathcal{P}$ on $%
\mathcal{G}_{0}$. Next we present the precise definitions in terms of the
isomorphism (\ref{isoT}). We also include the expressions of the associated
fundamental forms (see \cite{Alek,GG,GodoySalvaiMag,gk1, salvaimm, Salvai2}).

Given $\ell =\left[ \gamma \right] \in \mathcal{G}_{\kappa }$, the linear
complex structure $\mathcal{J}_{\ell }$ on $\mathfrak{J}_{\gamma }\cong
T_{\ell }\mathcal{G}_{\kappa }$, which was described geometrically in
Section \ref{sec:gintro}, is defined by 
\begin{equation}
\mathcal{J}_{\ell }\left( J\right) =\gamma ^{\prime }\times J,\hspace{0.5cm}%
\text{for }\ \ J\in \mathfrak{J}_{\gamma }\cong T_{\ell }\mathcal{G}_{\kappa
}\text{,}  \label{complex structure}
\end{equation}%
and the square norms of the metrics $g_{\times }$ and $g_{K}$ are given by 
%(we put $\left\Vert Z\right\Vert =\left\langle Z,Z\right\rangle $):
%For $J\in \mathfrak{J}_{\gamma }\cong T_{[\gamma ]}\mathcal{\mathcal{G}_{\kappa }}$,%
\begin{equation*}
g_{\times }\left( J,J\right) =\langle \gamma ^{\prime }\times J,J^{\prime
}\rangle _{\kappa }\text{\ \ \ \ \ \ and\ \ \ \ \ \ \ }g_{K}\left(
J,J\right) =|J|_{\kappa }^{2}+\kappa |J^{\prime }|_{\kappa }^{2}\text{.}
\end{equation*}%
Notice that by (\ref{equationJNormal}) the right hand sides are constant
functions, so the left hand sides are well defined. By polarization we have 
\begin{eqnarray}
2\,g_{\times }(I,J) &=&\langle I\times J^{\prime }+J\times I^{\prime
},\gamma ^{\prime }\rangle _{\kappa }\hspace{0.5cm}\text{for}\ \kappa =-1,0,1%
\text{;}  \label{gCroos} \\
g_{K}(I,J) &=&\langle I,J\rangle _{\kappa }+\kappa \langle I^{\prime
},J^{\prime }\rangle _{\kappa }\hspace{0.5cm}\text{for}\ \kappa =\pm 1\text{.%
}  \label{gkappa}
\end{eqnarray}%
The second one is the push down onto $\mathcal{G}_{\kappa }$ of the left
invariant pseudo-Riemannian metric on the Lie group Iso$_{0}(M_{\kappa })$
given at the identity by a multiple of the Killing form. It is Riemannian
for $\kappa =1$ and split for $\kappa =-1$. 
Proposition 4 in \cite{Salvai2}
provides a geometric interpretation for the metrics $g_{\times }$ and $g_{K}$
in the case $\kappa =-1$, with characterizations of space-like, time-like and null curves in $\mathcal{G}_{-1}$ in both cases.

The associated fundamental forms are given by 
\begin{eqnarray}
\omega _{\times }( I,J) &=&g_{\times }( \mathcal{J}( I) ,J) =\tfrac{1}{2}%
(\langle I^{\prime },J\rangle _{\kappa }-\langle I,J^{\prime }\rangle
_{\kappa })\text{,}  \label{omegaCross} \\
\omega _{K}( I,J) &=&g_{K}( \mathcal{J}( I) ,J) =\langle I\times J+\kappa
I^{\prime }\times J^{\prime },\gamma ^{\prime }\rangle _{\kappa }\text{.}
\label{omegaKplus}
\end{eqnarray}

We comment that $p^{\ast }\omega _{\times }$ is a constant multiple of $%
\Omega $, where $p:T^{1}M_{\kappa }\rightarrow \mathcal{G}_{\kappa }$, $%
v\mapsto \lbrack \gamma _{v}]$ is the canonical submersion and $\Omega $ is
the restriction to $T^{1}M_{\kappa }$ of the canonical symplectic form on $%
TM_{\kappa }$ (identified with the cotangent bundle $T^{\ast }M_{\kappa }$
through the Riemannian metric).

The bilinear form $g_{K}$ degenerates for $\kappa =0$, but we have the
canonical Poisson structure on $\mathcal{G}_{0}$ well defined at each $\ell$ by 
\begin{equation*}
\mathcal{P}(\ell )=J\wedge \mathcal{J}_{\ell }(J)\text{,}
\end{equation*}%
where $J$ is any parallel Jacobi field along $\ell $, orthogonal to it, with $\left\Vert
J\right\Vert \equiv 1$. Although no such section $\ell \mapsto J_{\ell }\in
T_{\ell }\mathcal{G}_{0}$ exists globally (otherwise, it would induce a unit
vector field on the 2-sphere), $\mathcal{P}$ is easily seen to be well
defined and smooth; the Schouten bracket $\left[ \mathcal{P},\mathcal{P}%
\right] $ vanishes, since the distribution on $\mathcal{G}_{0}$ induced by $%
\mathcal{P}$ is integrable. In fact, the symplectic leaves are the
submanifolds of parallel lines.

\section{The smoothness of the outer billiard map}

Here we establish notation and prove various technical lemmas, working
towards the proof of Theorem \ref{thm:diffeo}. Let $S$ be a closed smooth
surface in $M_{\kappa }$. Let $n$ be the inward-pointing unit normal vector
field on $S$. The complex structure $i$ on $S$ is defined by $iz=n(p)\times
z $ for $z\in T_{p}S$.

Given $w\in T^{1}M_{\kappa }$, we denote by $\gamma _{w}$ the unique
geodesic in $M_{\kappa }$ with initial velocity $w$. For $\kappa =0,-1$, let 
$T=\infty $, and for $\kappa =1$, let $T=\pi /2$. Define 
\begin{equation*}
F:\mathcal{M}\times (-T,T)\cong T^{1}S\times (-T,T)\rightarrow \mathcal{G}%
_{\kappa }\text{,\ \ \ \ \ \ \ \ \ \ }F(u,t)=[\gamma _{u_{t}}]\text{,}
\end{equation*}%
where $u_{t}$ is the parallel transport on $M_{\kappa }$ of $u$ along $%
\gamma _{iu}$ between $0$ and $t$; see Figure \ref{fig:ut}. Let $F_{+}$ and $F_{-}$ denote the
restrictions of $F$ to $T^{1}S\times (0,T)$ and $T^{1}S\times (-T,0)$, respectively. 

\textbf{\begin{figure}[h!t]
		\centerline{
			\includegraphics[width=2in]
			{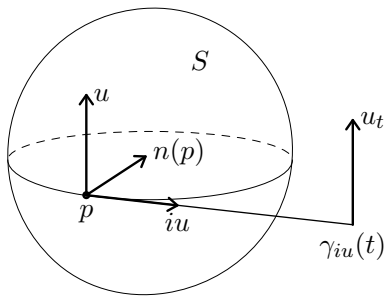}}
		\caption{The parallel transport of $u$ along $\gamma _{iu}$ between $0$ and $t$}
		\label{fig:ut}
\end{figure}}

By the construction in Section \ref{sec:obdef}, the outer billiard map $B:%
\mathcal{U}\rightarrow \mathcal{U}$ is equal to the composition 
\begin{equation}\label{composition}
B=F_{+}\circ g\circ \left( F_{-}\right) ^{-1}\text{,}
\end{equation}%
where $g:T^{1}S\times (-T,0)\rightarrow T^{1}S\times (0,T)$ is defined by $%
g(u,t)=(u,-t)$. Clearly $F$ is a smooth function and $g$ is a
diffeomorphism, and so the proof of Theorem \ref{thm:diffeo} reduces to
showing that $F_{\pm }$ are diffeomorphisms. Since they are bijections, we
must show that 
\begin{equation*}
(dF)_{\left( u,t\right) }:T_{u}T^{1}S\times T_{t}\mathbb{R}\rightarrow
T_{F(u,t)}\mathcal{G}_{\kappa }\cong \mathfrak{J}_{\gamma _{u_{t}}}\text{,}
\end{equation*}%
is nonsingular for all $u\in T^{1}S$ and $0\neq \left\vert t\right\vert <T$.
To verify this, we will compute the differential with respect to certain
canonical bases which we introduce next.

Given an oriented geodesic $\ell \in \mathcal{M}$, there are three
perturbations of $\ell $ which stay in $\mathcal{M}$: one which skates along 
$S$ in the direction of $\ell $, one which parallel transports $\ell $ along 
$S$ in the direction orthogonal to $\ell $, and one which rotates $\ell $,
maintaining the point of tangency. These three perturbations may be thought
of as generating the tangent space $T_{\ell }\mathcal{M}$. We formalize this
intuitive idea below, via the natural identification of $\mathcal{M}$ with $%
T^{1}S$.

In what follows we fix $u\in
T^{1}_pS$ and $t\neq 0$ and denote $v=iu$. 

Given a unit tangent vector $z\in T^{1}S$, we call $\sigma _{z}$ the
geodesic of $S$ with initial velocity $z$. Let $\pi :TS\rightarrow S$ be 
the canonical projection and let $\mathcal{K}%
_{u}:T_{u}TS\rightarrow T_{p}S$ be the connection operator, which is well
defined as follows: Given $\xi \in T_{u}TS$, let $U:\left( -\delta ,\delta
\right) \rightarrow TS$ be a smooth curve with $U\left( 0\right) =u$ and
initial velocity $\xi $. Then $K_{u}\left(
\xi \right) =\frac{DU}{dt}\left( 0\right) $, where $\frac{D}{dt}$ is the
covariant derivative along the foot-point curve $\pi \circ U $.

\begin{lemma}
\label{lem:w} For $m=1,2,3$, let $w_{m}:\mathbb{R}\rightarrow T^{1}S$ be the
curve defined by 
\begin{equation}
w_{1}(s)=\sigma _{u}^{\prime }(s),\hspace{0.5cm}w_{2}(s)=\tau _{0,s}^{\sigma
_{v}}(u),\hspace{0.5cm}w_{3}(s)=\cos s~u+\sin s~v,  \label{eqn:w}
\end{equation}%
where $\tau _{0,s}^{\sigma }$ denotes the parallel transport on $S$ along $%
\sigma $ between $0$ and $s$. Then $\left\{ w_{1}^{\prime }(0),w_{2}^{\prime
}(0),w_{3}^{\prime }(0)\right\} $ is a basis of $T_{u}T^{1}S$.
\end{lemma}

\begin{proof}
 We claim that
under the linear isomorphism 
\begin{equation*}
\varphi _{u}\colon T_{u}TS\rightarrow T_{p}S\times T_{p}S,\ \ \ \ \ \ \ \ \
\varphi _{u}(\xi )=(d\pi _{u}\xi ,\mathcal{K}_{u}\xi )
\end{equation*}%
(see for instance \cite{besse}), $w_{1}^{\prime }(0),w_{2}^{\prime
}(0),w_{3}^{\prime }(0)$ are mapped, respectively, to the linearly
independent vectors $(u,0)$, $(v,0)$ and $(0,v)$. We compute 
\begin{equation*}
d\pi _{u}(w_{1}^{\prime }(0))=\left. \tfrac{d}{ds}\right\vert _{0}\pi
(w_{1}(s))=\left. \tfrac{d}{ds}\right\vert _{0}\sigma _{u}(s)=u\text{,}
\end{equation*}%
and, by definition of $\mathcal{K}_{u}$, 
\begin{equation*}
\mathcal{K}_{u}(w_{1}^{\prime }(0))=\left. \tfrac{D}{ds}\right\vert
_{0}w_{1}(s)=\left. \tfrac{D}{ds}\right\vert _{0}\sigma _{u}^{\prime }(s)=0%
\text{.}
\end{equation*}%
Hence, $\varphi _{u}(w_{1}^{\prime }(0))=(u,0)$. The other cases are similar.
\end{proof}

We consider the basis $\mathcal{B}_t=\left\{ W_{1},W_{2},W_{3},W_{4}\right\} $
of $T_{u}T^{1}S\times T_{t}\mathbb{R}$, with%
\begin{equation}
W_{m}=(w_{m}^{\prime }(0),0)\text{, \ \ \ \ for }m=1,2,3\text{\ \ \ \ \ \ \
and\ \ \ \ \ \ \ }W_{4}=\left( 0,\left. \tfrac{d}{ds}\right\vert _{t}\right) 
\text{,}  \label{basisB}
\end{equation}%
where $w_{m}$ are the curves defined in (\ref{eqn:w}). Now, the
image of $W_{m}$ by $(dF)_{(u,t)}$ is a tangent vector to $\mathcal{G}%
_{\kappa }$ at $F\left( u,t\right) $, and so by (\ref{isoT}), it corresponds
to a Jacobi field along $\gamma _{u_{t}}$ in $\mathfrak{J}_{\gamma _{u_{t}}}$%
, which we call $J_{m}$. We state this in the following proposition, whose
proof is straightforward from the definitions.

\begin{proposition}
For $m=1,\dots ,4$ we have%
\begin{equation*}
(dF)_{(u,t)}\left( W_{m}\right) =T_{\gamma _{u_{t}}}\left( J_{m}\right) 
\text{,}
\end{equation*}%
where $J_{m}$ is the normal component of the Jacobi field arising from the
geodesic variations of $\gamma _{u_{t}}$ given by 
\begin{equation*}
\left( s,r\right) \mapsto \gamma _{(w_{m}(s))_{t}}\left( r\right)
\end{equation*}%
for $m=1,2,3$, and $\left( s,r\right) \mapsto \gamma _{u_{t+s}}\left(
r\right) $ for $m=4$.
\end{proposition}

We need $J_{m}$ explicitly. We consider the parametrized surface 
\begin{equation*}
f_{m}:\mathbb{R}^{2}\rightarrow M_{\kappa }\text{,\ \ \ \ \ \ \ \ }%
f_{m}(r,s)=\gamma _{iw_{m}(s)}(r)\text{.}
\end{equation*}%
In particular, $f_{m}(r,0)=\gamma _{v}(r)$. We write 
\begin{equation*}
\sigma _{m}=\pi \circ w_{m}=f_{m}(0,\cdot )\text{,}
\end{equation*}%
so $\sigma _{1}=\sigma _{u}$, $\sigma _{2}=\sigma _{v}$, and $\sigma
_{3}\equiv p$. Now we can describe the initial conditions of $J_{m}$ in
terms of some vector fields along $f_{m}$. 

\begin{proposition}\label{prop3.3}
For $m=1,2,3$ we have 
\begin{equation*}
J_{m}(0)=K_{m}(t)-\left\langle K_{m}(t),u_{t}\right\rangle _{\kappa }u_{t}%
\hspace{0.5cm}\text{and}\hspace{0.5cm}J_{m}^{\prime }(0)=\left. \tfrac{D}{ds}%
\right\vert _{0}Z_{m}(t,s)\text{,}
\end{equation*}%
where $K_{m}$ is the Jacobi vector field along $\gamma _{v}$ associated with
the geodesic variation $f_{m}$, that is, 
\begin{equation*}
K_{m}(r)=\left. \tfrac{d}{ds}\right\vert _{0}f_{m}(r,s)\text{,}
\end{equation*}%
and $Z_{m}$ is the vector field along the surface $f_{m}$ obtained by
parallel transporting $w_{m}(s)$ on $M_{\kappa }$ along the geodesic $\gamma
_{iw_{m}(s)}$ from $0$ to $r$, that is: 
\begin{equation*}
Z_{m}(r,s)=P_{0,r}^{\gamma _{iw_{m}(s)}}w_{m}(s)\text{.}
\end{equation*}

Also, 
\begin{equation*}
J_{4}(0)=v_{t}\hspace{0.5cm}\ \ \text{and}\hspace{0.5cm}\ \ J_{4}^{\prime
}(0)=0\text{,}
\end{equation*}%
where $v_{t}$ is the parallel transport on $M_{\kappa }$ of $v$ between $0$
and $t$ along $\gamma _{v}$.
\end{proposition}

\begin{proof} Let $\bar{J}_{m}\left( r\right) =\left. \frac{d}{ds}\right\vert _{0}\gamma
_{(w_{m}(s))_{t}}\left( r\right) $. By (\ref{projectionK}), $J_{m}\left( r\right)
=\left( \bar{J}_{m}\left( r\right) \right) ^{N}$. We compute%
\begin{equation*}
	\bar{J}_{m}\left( 0\right) =\left. \frac{d}{ds}\right\vert _{0}\gamma
	_{(w_{m}(s))_{t}}\left( 0\right) =\left. \frac{d}{ds}\right\vert _{0}\gamma
	_{iw_{m}(s)}(t)=\left. \frac{d}{ds}\right\vert _{0}f_{m}(t,s)=K_{m}\left(
	t\right) \text{.}
\end{equation*}%
Then $J_{m}\left( 0\right) $ is as stated. Also, ${J}_{m}^{\prime
}\left( 0\right) =\bar {J}_{m}^{\prime }\left( 0\right)$, since the geodesics in
the variation have unit speed. Now,
\begin{equation*}
\bar {J}_{m}^{\prime }\left( 0\right)=\left.\frac{D}{\partial r}\right|_0 \left.\frac{\partial}{\partial s}\right|_0 \gamma
_{(w_{m}(s))_{t}}\left( r\right)= \left.\frac{D}{\partial s}\right|_0 \left.\frac{\partial}{\partial r}\right|_0 \gamma
_{(w_{m}(s))_{t}}\left( r\right)=\left.\frac{D}{\partial s}\right|_0 (w_m (s))_t=\left. \tfrac{D}{ds}%
\right\vert _{0}Z_{m}(t,s)\text{.}
\end{equation*}
The validity of the remaining assertions, involving $J_4$, follows from similar (simpler) arguments.
\end{proof}

\bigskip

We require explicit formulas for $K_{m}$ and $\left. \tfrac{D}{ds}%
\right\vert _{0}Z_{m}(t,s)$, which are vector fields along $\gamma _{v}$. In
the next three lemmas we compute their coordinates with respect to the basis 
$\left\{ u_{r},v_{r},n_{r}\right\} $ of $T_{\gamma _{v}\left( r\right)
}M_{\kappa }$, where $u_{r}$, $v_{r}$, and $n_{r}$ are obtained by parallel
transporting $u$, $v$, and $n(p)$ along $\gamma _{v}$, between $0$ and $r$.

Given $p\in S$, the shape operator $A_{p}:T_{p}S\rightarrow T_{p}S$ is
defined by 
\begin{equation}
A_{p}\left( x\right) =-\nabla _{x}n\text{,}  \label{shapeop}
\end{equation}%
where $\nabla $ denotes the Levi-Civita connection of $M_{\kappa }$. In what
follows we assume that $A_{p}$ is positive definite at each $p\in S$ (that
is, $S$ is quadratically convex).

We consider the matrix of $A_{p}$ with respect to the orthonormal basis $%
\left\{ u,v\right\} $ and call $b_{ij}$ its entries, that is, $\left[ A_{p}%
\right] _{\left\{ u,v\right\} }=\left( 
\begin{array}{cc}
b_{11} & b_{12} \\ 
b_{21} & b_{22}%
\end{array}%
\right) $, with $b_{12}=b_{21}$.

\begin{lemma}
\label{KM} For $m=1,2,3$, the Jacobi vector field $K_{m}$ along $\gamma _{v}$
is given by 
\begin{equation}
\begin{array}{l}
K_{1}(r)=c_{\kappa }(r)u_{r}+b_{21}s_{\kappa }(r)n_{r}, \\ 
K_{2}(r)=v_{r}+b_{22}s_{\kappa }(r)n_{r}, \\ 
K_{3}(r)=-s_{\kappa }(r)u_{r}.%
\end{array}
\label{KmFormula}
\end{equation}
\end{lemma}

\begin{proof}
We compute the initial values of $K_{m}$ and $K_{m}^{\prime }$ and use (\ref%
{JacobiK}). We write down the details for $m=1$. The other cases are
similar. We compute 
\begin{equation*}
K_{1}(0)=\left. \tfrac{d}{ds}\right\vert _{0}f_{1}(0,s)=\left. \tfrac{d}{ds}%
\right\vert _{0}\sigma _{1}(s)=u
\end{equation*}%
and 
\begin{equation*}
K_{1}^{\prime }(0)=\left. \tfrac{D}{\partial r}\right\vert _{0}\left. \tfrac{%
\partial }{\partial s}\right\vert _{0}f_{1}(r,s)=\left. \tfrac{D}{\partial s}%
\right\vert _{0}\left. \tfrac{\partial }{\partial r}\right\vert
_{0}f_{1}(r,s)=\left. \tfrac{D}{\partial s}\right\vert _{0}iw_{1}(s)\text{.}
\end{equation*}%
We compute the coordinates of $K_{1}^{\prime }(0)$ with respect to the
orthonormal basis $\{u,v,n(p)\}$ of $T_{p}M_{\kappa }$. To obtain $%
\left\langle K_{1}^{\prime }(0),n(p)\right\rangle _{\kappa }$, we observe
that $\left\langle iw_{1}(s),n(\sigma _{1}(s))\right\rangle _{\kappa }=0$
for all $s$. Hence, 
\begin{equation*}
\left\langle \left. \tfrac{D}{\partial s}\right\vert
_{0}iw_{1}(s),n(p)\right\rangle _{\kappa }=-\left\langle iu,\left. \tfrac{D}{%
\partial s}\right\vert _{0}n(\sigma _{1}(s))\right\rangle _{\kappa
}=-\left\langle v,\nabla _{u}n\right\rangle _{\kappa }=\left\langle
v,A_{p}(u)\right\rangle _{\kappa }=\,b_{21}\text{.}
\end{equation*}%
In the same way, $\left\langle K_{1}^{\prime }(0),u\right\rangle _{\kappa
}=0=\left\langle K_{1}^{\prime }(0),v\right\rangle _{\kappa }$. Therefore, $%
K_{1}^{\prime }(0)=b_{21}n(p)$. Notice that $K_{m}$ is not necessarily
orthogonal to $\gamma _{v}^{\prime }$.
\end{proof}

For the sake of simplicity of notation we denote by $Y_{m}\left( r\right)
=\left. \tfrac{D}{ds}\right\vert _{0}Z_{m}(r,s)$.

\begin{lemma}
\label{ZY}For $m=1,2,3$, the vector field $Y_{m}$ along $\gamma _{v}$ is
given by 
\begin{equation*}
Y_{1}(r) =\kappa s_{\kappa }(r)v_{r}+b_{11}n_{r}\text{,} \ \ \ \ \ \ \ \ 
Y_{2}(r)=b_{12}n_{r}\text{,} \ \ \ \ \ \ \ \ 
Y_{3}(r)=c_{\kappa }(r)v_{r}\text{.}
\end{equation*}
\end{lemma}

Before proving the lemma we introduce the vector field $N_{m}$ along the
surface $f_{m}$ obtained by parallel transporting $n(\sigma _{m}(s))$ on $%
M_{\kappa }$ along the geodesic $\gamma _{iw_{m}(s)}$ from $0$ to $r$, that
is, 
\begin{equation*}
N_{m}(r,s)=P_{0,r}^{\gamma _{iw_{m}(s)}}n(\sigma _{m}(s))\text{.}
\end{equation*}

\begin{lemma}
\label{Nm} Let $\zeta _{m}\left( r\right) =\left\langle \left. \tfrac{D}{ds}%
\right\vert _{0}N_{m}(r,s),u_{r}\right\rangle _{\kappa }$. Then $\zeta
_{1}\equiv -b_{11}$, $\zeta _{2}\equiv -b_{12}$, and $\zeta _{3}\equiv 0$.
\end{lemma}

\begin{proof}
We compute%
\begin{eqnarray*}
\zeta _{m}^{\prime }(r) &=&\tfrac{d}{dr}\left\langle \left. \tfrac{D}{ds}%
\right\vert _{0}N_{m}(r,s),u_{r}\right\rangle _{\kappa }=\left\langle \tfrac{%
D}{dr}\left. \tfrac{D}{ds}\right\vert _{0}N_{m}(r,s),u_{r}\right\rangle
_{\kappa } \\
&=&\left\langle \left. \tfrac{D}{ds}\right\vert _{0}\tfrac{D}{dr}%
N_{m}(r,s)+R_{\kappa }\left( \tfrac{\partial f_{m}}{\partial r}(r,0),\tfrac{%
\partial f_{m}}{\partial s}(r,0)\right) N_{m}(r,0),u_{r}\right\rangle
_{\kappa } \\
&=&\left\langle R_{\kappa }(v_{r},K_{m}(r))n_{r},u_{r}\right\rangle _{\kappa
}=\kappa \left\langle \left\langle n_{r},v_{r}\right\rangle _{\kappa
}K_{m}(r)-\left\langle n_{r},K_{m}(r)\right\rangle _{\kappa
}v_{r},u_{r}\right\rangle _{\kappa }=0\text{,}
\end{eqnarray*}%
by (\ref{curvatureK}). Hence, $\zeta _{m}$ is constant, equal to 
\begin{equation*}
\zeta _{m}(0)=\left\langle \left. \tfrac{D}{ds}\right\vert
_{0}N_{m}(0,s),u_{0}\right\rangle _{\kappa }=\left\langle \nabla _{\sigma
_{m}^{\prime }(0)}n,u\right\rangle _{\kappa }=-\left\langle A_{p}(\sigma
_{m}^{\prime }(0)),u\right\rangle _{\kappa }.
\end{equation*}%
Now, the assertions follow from the definition of $\sigma _{m}$ and the
values of the entries of the matrix $\left[ A_{p}\right] _{\left\{
u,v\right\} }$.
\end{proof}

\begin{proof}[Proof of Lemma \protect\ref{ZY}]
Observe that $\left\langle Z_{m},Z_{m}\right\rangle _{\kappa }$, $%
\left\langle Z_{m},N_{m}\right\rangle _{\kappa }$ and $\langle Z_{m},\frac{%
\partial f_{m}}{\partial r}\rangle _{\kappa }$ are constant functions of the
second variable $s$. Hence we can compute the components $Y_{m}(r)$ with
respect to the basis $\left\{ u_{r},v_{r},n_{r}\right\} $ of $T_{\gamma
_{v}\left( r\right) }M_{\kappa }$ as follows: 
\begin{equation*}
\left\langle Y_{m}(r),u_{r}\right\rangle _{\kappa }=\left\langle \left. 
\tfrac{D}{ds}\right\vert _{0}Z_{m}(r,s),Z_{m}(r,0)\right\rangle _{\kappa }=0.
\end{equation*}%
Also 
\begin{eqnarray*}
\left\langle Y_{m}(r),v_{r}\right\rangle _{\kappa } &=&\left\langle \left. 
\tfrac{D}{ds}\right\vert _{0}Z_{m}(r,s),v_{r}\right\rangle _{\kappa
}=-\left\langle Z_{m}(r,0),\left. \tfrac{D}{ds}\right\vert _{0}\gamma
_{iw_{m}\left( s\right) }^{\prime }(r)\right\rangle _{\kappa } \\
&=&-\left\langle u_{r},\tfrac{D}{dr}\left. \tfrac{d}{ds}\right\vert
_{0}\gamma _{iw_{m}(s)}(r)\right\rangle _{\kappa }=-\left\langle u_{r},%
\tfrac{D}{dr}K_{m}(r)\right\rangle _{\kappa }\text{.}
\end{eqnarray*}%
By Lemma \ref{KM}, we have 
\begin{equation*}
-\left\langle u_{r},K_{m}^{\prime }(r)\right\rangle _{\kappa }=%
\begin{cases}
\kappa s_{\kappa }(r), & \text{if }m=1\text{,} \\ 
0, & \text{if }m=2\text{,} \\ 
c_{\kappa }(r), & \text{if }m=3\text{.}%
\end{cases}%
\end{equation*}%
In the same way, 
\begin{equation*}
\left\langle Y_{m}(r),n_{r}\right\rangle _{\kappa }=\left\langle \left. 
\tfrac{D}{ds}\right\vert _{0}Z_{m}(r,s),N_{m}(r,0)\right\rangle _{\kappa
}=-\left\langle Z_{m}(r,0),\left. \tfrac{D}{ds}\right\vert
_{0}N_{m}(r,s)\right\rangle _{\kappa }=-\zeta _{m}(r)\text{,}
\end{equation*}%
with $\zeta _{m}$ as in Lemma \ref{Nm}.
\end{proof}

With the computational lemmas above, we can present the proof of Theorem \ref%
{thm:diffeo}.

\begin{proof}[Proof of Theorem \protect\ref{thm:diffeo}]
To complete the proof of Theorem \ref{thm:diffeo}, it remains to show that 
\begin{equation*}
(dF)_{(u,t)}:T_{u}T^{1}S\times T_{t}\mathbb{R}\rightarrow T_{F(u,t)}\mathcal{G}%
_{\kappa }\cong \mathfrak{J}_{\gamma _{u_{t}}}
\end{equation*}%
is nonsingular for all $u\in T^{1}S$ and $0\neq |t|<T$. To verify this, we
compute the matrix of $(dF)_{(u,t)}$ with respect to the bases $\mathcal{B}_t$
(given in (\ref{basisB})) of $T_{u}T^{1}S\times T_{t}\mathbb{R}$ and $%
\mathcal{E}_{t}=\left\{ E_{1}^{t},\dots ,E_{4}^{t}\right\} $ of $\mathfrak{J}%
_{\gamma _{u_{t}}}$, where $E_{m}^{t}$ are the Jacobi fields along $\gamma
_{u_{t}}$ whose initial conditions are 
\begin{equation}
\begin{array}{cccc}
E_{1}^{t}(0)=0, & E_{2}^{t}(0)=n_{t}, & E_{3}^{t}(0)=0, & E_{4}^{t}(0)=v_{t},
\\ 
&  &  &  \\ 
(E_{1}^{t})^{\prime }(0)=n_{t}, & (E_{2}^{t})^{\prime }(0)=0, & 
(E_{3}^{t})^{\prime }(0)=v_{t}, & (E_{4}^{t})^{\prime }(0)=0.%
\end{array}
\label{base}
\end{equation}%
By Proposition \ref{prop3.3} and Lemmas \ref{KM} and \ref{ZY}, we have 
\begin{equation*}
J_{m}(0)=%
\begin{cases}
b_{21}s_{\kappa }(t)n_{t}, & m=1\text{,} \\ 
v_{t}+b_{22}s_{\kappa }(t)n_{t}, & m=2\text{,} \\ 
0, & m=3\text{,} \\ 
v_{t}, & m=4\text{,}%
\end{cases}%
\hspace{0.5cm}\text{and}\hspace{0.5cm}J_{m}^{\prime }(0)=%
\begin{cases}
\kappa s_{\kappa }(t)v_{t}+b_{11}n_{t}, & m=1\text{,} \\ 
b_{12}n_{t}, & m=2\text{,} \\ 
c_{\kappa }(t)v_{t}, & m=3\text{,} \\ 
0, & m=4\text{.}%
\end{cases}%
\end{equation*}%
Hence, calling $C_t$ the matrix of $(dF)_{(u,t)}$ with respect to the bases $%
\mathcal{B}_t$ and $\mathcal{E}_{t}$, we obtain that 
\begin{equation}\label{Ct}
C_t=\left( 
\begin{array}{cccc}
b_{11} & b_{12} & 0 & 0 \\ 
b_{21}s_{\kappa }(t) & b_{22}s_{\kappa }(t) & 0 & 0 \\ 
\kappa s_{\kappa }(t) & 0 & c_{\kappa }(t) & 0 \\ 
0 & 1 & 0 & 1%
\end{array}%
\right) .
\end{equation}

Therefore, $\det C_t=c_{\kappa }(t)s_{\kappa }(t)b$, where $b= \det
\left(A_p\right)$. Since, by hypothesis, $0 \neq |t| < T$ and $%
[A_p]_{\{u,v\}} $ is definite, we have that $C_t$ is nonsingular.
\end{proof}

\begin{proof}[Proof of Proposition \ref{notSmooth}] By the reflection invariance, $B$ preserves the oriented
	lines orthogonal to the plane $P=\left\{ \left( x,y,z\right) \mid
	y=0\right\} $ and induces in the obvious manner the outer billiard map $%
	\widetilde{B}$ on $P$ determined by the closed strictly convex curve $\gamma 
	$ with image $S\cap P$, which includes the graph of $f$. Accordingly, we
	identify $P$ and the set of oriented lines orthogonal to it with $\mathbb{R}%
	^{2}$.
	
	Given a small $s\geq 0$, we next compute $\widetilde{B}\left( -1,s\right) $.
	Let $\ell _{s}$ be the straight line passing through $\left( -1,s\right) $
	tangent to $\gamma $ at $\left( x_{s},f\left( x_{s}\right) \right) $, with $%
	-1<x_{s}\leq 0$. Then $\ell _{s}$ can be parametrized by $t\mapsto
	l_{s}\left( t\right) =\left( x_{s},f\left( x_{s}\right) \right) +t\left(
	1,f^{\prime }\left( x_{s}\right) \right) $ and there exists a unique $t_{s}$
	such that $l_{s}\left( t_{s}\right) =\left( -1,s\right) $. We have%
	\begin{equation}
		x_{0}=0\text{,\ \ \ \ \ \ }x_{s}+t_{s}=-1\text{\ \ \ \ \ and\ \ \ \ \ }%
		f\left( x_{s}\right) +t_{s}f^{\prime }\left( x_{s}\right) =s\text{.}
		\label{*}
	\end{equation}
		Hence,%
	\begin{equation*}
		\widetilde{B}\left( -1,s\right) =l_{s}\left( -t_{s}\right) =\left(
		x_{s}-t_{s},f\left( x_{s}\right) -t_{s}f^{\prime }\left( x_{s}\right)
		\right) =\left( 2x_{s}+1,2f\left( x_{s}\right) -s\right) .
	\end{equation*}%
	Suppose that $\widetilde{B}$ is smooth, then so are $s\mapsto x_{s}$ and $%
	s\mapsto t_{s}$. We compute the right derivative at $s=0$ of both sides of
	the last equation in (\ref{*}) and obtain%
	\begin{equation*}
		0=f^{\prime }\left( 0\right) x_{0}^{\prime }+t_{0}^{\prime }f^{\prime
		}\left( 0\right) +t_{0}f^{\prime \prime }\left( 0\right) x_{0}^{\prime }=1%
		\text{,}
	\end{equation*}%
	a contradiction.
\end{proof}

\section{The K\"{a}hler formulation of the outer billiard map}

\label{sec:kahlerdef}

In order to prove Theorem \ref{thm:equiv} we need the presentation of $%
\mathcal{G}_{\kappa }$ as a symmetric homogeneous space. The details of the
following description can be found for instance in \cite{GodoySalvaiMag}.
For $\kappa =\pm 1$, we consider the standard presentation of $M_{\kappa }$
as a submanifold of $\mathbb{R}^{4}$: If $\{e_{0},e_{1},e_{2},e_{3}\}$ is
the canonical basis of $\mathbb{R}^{4}$, then $M_{k}$ is the connected
component of $e_{0}$ of the set 
\begin{equation*}
\left\{ (t,x,y,z)\in \mathbb{R}^{4}\,|\,\kappa
t^{2}+x^{2}+y^{2}+z^{2}=\kappa \right\} .
\end{equation*}

Let $G_{\kappa }$ be the identity component of the isometry group of $%
M_{\kappa }$, that is, $G_{1}=SO_{4}$ and $G_{-1}=O_{o}\left( 1,3\right) $.
The group $G_{\kappa }$ acts smoothly and transitively on $\mathcal{G}%
_{\kappa }$ as follows: $g\cdot \lbrack \gamma ]=[g\circ \gamma ]$. Let $%
\gamma _{o}$ be the geodesic in $M_{\kappa }$ with $\gamma _{o}(0)=e_{0}$
and initial velocity $e_{1}\in T_{e_{0}}M_{\kappa }$ and let $H_{\kappa }$
be the stabilizer group of $[\gamma _{o}]$ in $G_{\kappa }$. Then there
exists a diffeomorphism $\phi :G_{\kappa }/H_{\kappa }\rightarrow \mathcal{G}%
_{\kappa }$, given by $\phi (gH_{\kappa })=g\cdot \lbrack \gamma _{o}\rbrack$%
.

The Killing form of Lie\thinspace $(G_{\kappa })$ provides $G_{\kappa }$
with a bi-invariant metric and thus there exists a unique pseudo-Riemannian
metric $\tilde{g}_{K }$ on $G_{\kappa }/H_{\kappa }$ such that the canonical
projection $\pi :G_{\kappa }\rightarrow G_{\kappa }/H_{\kappa }$ is a
pseudo-Riemannian submersion. The diffeomorphism $\phi $ turns out to be an
isometry onto $\mathcal{G}_{\kappa }$ endowed with a constant multiple of
the metric $g_{K}$ defined in (\ref{gkappa}).

Besides, it is well known that $(G_{\kappa }/H_{\kappa },\tilde{g}_{K}) $ is
a pseudo-Riemannian symmetric space. In particular, if Lie$\,\left(
G_{\kappa }\right) =$ Lie$\,\left( H_{\kappa }\right) \oplus \mathfrak{p}%
_{\kappa }$ is the Cartan decomposition determined by $[\gamma _{o}]$, then
for any $Z\in \mathfrak{p}_{\kappa }$ the curve $t\mapsto \exp \left(
tZ\right) H_{\kappa }$ is a geodesic of $G_{\kappa }/H_{\kappa }$.

\bigskip

\begin{proof}[Proof of Theorem \protect\ref{thm:equiv}]
To see that $B=B^{\prime }$, since $\mathcal{M}\cong T^{1}S$, we only need to
show that for $u\in T_{p}^{1}S$, the curve $\Gamma (t)=:\left[ \gamma
_{u_{t}}\right] $ is the geodesic in $(\mathcal{G}_{\kappa },g_{K})$ with
initial velocity $\mathcal{J}\nu(u)$, where $\nu(u)$ is the outward pointing
normal vector of $\mathcal{M}$ at $[\gamma _{u}]$. First we verify that $%
\Gamma ^{\prime }(0)=\mathcal{J}\nu(u)$ and afterwards that $\Gamma $ is a
geodesic of $\mathcal{G}_{\kappa }$.

The initial velocity of $\Gamma$ corresponds, via the isomorphism $T_{\gamma
_{u}}$ of (\ref{isoT}), with the Jacobi field along $\gamma _{u}$ given by 
\begin{equation*}
J(s)=\left. \tfrac{d}{dt}\right\vert _{0}\gamma _{u_{t}}(s)\text{.}
\end{equation*}%
A straightforward computation shows that $J$ is determined by the conditions 
$J(0)=iu$ and $J^{\prime }(0)=0$.

On the other hand, let $I \in \mathfrak{J}_{\gamma_u}$ be the Jacobi field
given by the initial conditions $I(0)=-n(p)$ and $I^{\prime }(0) = 0$. We
claim that after the identification with $T_{[\gamma_u]}\mathcal{G}_\kappa$, 
$I$ corresponds to the unit outward-pointing normal vector field $\nu$ on $%
\mathcal{M}$. Indeed, 
\begin{equation*}
g_{K}( I,I) =\langle I,I\rangle _{\kappa }+\kappa \langle I^{\prime
},I^{\prime }\rangle _{\kappa }=\left( -1\right) ^{2}\left\vert
n(p)\right\vert _{\kappa }^{2}=1,
\end{equation*}
and for $K \in T_{\left[ \gamma _{u}\right] }\mathcal{M}$, $K(0) \in
T_{\gamma _{u}\left( 0\right) }S$ by Lemma \ref{tangentMscript}, and so 
\begin{equation*}
g_{K}(I,K)=-\langle n(p),K( 0) \rangle _{\kappa }=0.
\end{equation*}
Also, $\nu(u)$ points to $\mathcal{U}$ since $n$ is the inward-pointing unit normal vector field of $S$.

Now, by the definition of the complex structure $\mathcal{J}$ in (\ref%
{complex structure}), the identity $\mathcal{J}_{[\gamma _{u}]}( \nu_{[\gamma
_{u}]}) =\Gamma ^{\prime }( 0) $ translates into $J=\gamma _{u}^{\prime
}\times I$, which holds since $J(0) = iu = n(p) \times u = \gamma^{\prime
}_u(0) \times I(0)$ and $J^{\prime }(0) = 0 = I^{\prime }(0)$.

Next we show that $\Gamma $ is a geodesic. By homogeneity, we may suppose
that $p=e_{0}$, the inward pointing unit normal vector of $S$ at $e_{0}$ is $%
e_{3}$ and $u=e_{1}$. Hence $iu=e_{2}$ and $u_{t}=e_{1}\in T_{\gamma \left(
t\right) }M_{\kappa }$ where $\gamma ( t) =c_{\kappa }( t)e_{0}+s_{\kappa }(
t)e_{2}$.

Let $Z$ be the linear transformation of $\mathbb{R}^{4}$ defined by $Z(
e_{0}) =e_{2}$, $Z( e_{2}) =-\kappa e_{0}$ and $Z( e_{1}) =Z( e_{3}) =0$. It
is easy to verify that $Z\in $ Lie$(G_{\kappa })$ and $\exp ( tZ) %
\left[ \gamma _{u}\right] =\left[ \gamma _{u_{t}}\right] $ for all $t$. Now,
one can see in the preliminaries of \cite{GodoySalvaiMag} (page 752) that $%
Z\in \mathfrak{p}_{\kappa }$, and so $\Gamma $ is a geodesic by the
properties of symmetric spaces presented above.
\end{proof}

\begin{proof}[Proof of Proposition \ref{gCross}]
	Suppose that $\ell =\left[ \gamma _{u}\right] $ with $\gamma
_{u}\left( 0\right) =p\in S$. The Jacobi field $I$ along $\gamma _{u}$ with $%
I\left( 0\right) =0$\ \ and \ \ $I^{\prime }\left( 0\right) =iu$ spans the
normal space to $T_{\ell }\mathcal{M}$ and is null (see (\ref{gCroos})). 
By Lemma \ref{tangentMscript}, $I$ is also tangent to $\mathcal{M}$ at $\ell $
(this shows, in particular, that $g_{\times}$ degenerates on $T_{\ell }\mathcal{M}
$). Now,%
\begin{equation*}
	\left( \mathcal{J}I\right) \left( 0\right) =0\text{\ \ \ and\ \ \ \ }\left( 
	\mathcal{J}I\right) ^{\prime }\left( 0\right) =n_{p}
\end{equation*}%
(where $n$ is the inward-pointing unit normal vector field of $S$, as
before). Again by Lemma \ref{tangentMscript}, $\mathcal{J}I\in T_{\ell }\mathcal{M}$.

The geodesic $\Gamma $ in $\mathcal{G}_{\kappa }$ with $\Gamma \left(
0\right) =\ell $ and $\Gamma ^{\prime }\left( 0\right) =\mathcal{J}I$
consists of oriented geodesics in $M_{\kappa }$ rotating around $p$ in the
totally geodesic surface orthogonal to $S$ containing the image of $\gamma _{u}
$, that is,%
\begin{equation*}
	\Gamma \left( t\right) =\left[ \gamma _{\cos t~u+\sin t~n_{p}}\right] 
\end{equation*}%
(this can be verified with computations similar to those we made at the end
of the proof of the theorem above). Hence, for each $t$, the oriented
geodesic $\Gamma \left( t\right) $ intersects $S$ and so, the image of $%
\Gamma $ is disjoint from $\mathcal{U}$. Since any normal $N$ is multiple of 
$I$, the proof concludes.
\end{proof}

\section{The symplectic properties of the outer billiard map}

\begin{proof}[Proof of Theorem \protect\ref{thm:symp} (a)]
As in  (\ref{composition}), we write 
\begin{equation*}
B=F_{+}\circ g\circ F_{-}^{-1}.
\end{equation*}
Given $\ell \in \mathcal{U}$, suppose that $\ell =F_{-}(u,-t)$ for some $%
0<t<T$ . We compute the matrix of $(dB)_{\ell }$ with respect to the canonical
bases $\mathcal{E}_{-t}$ and $\mathcal{E}_{t}$ of $\mathfrak{J}_{\gamma
_{u_{-t}}}$ and $\mathfrak{J}_{\gamma _{u_{t}}}$ as in (\ref{base}),
respectively, obtaining 
\begin{equation}\label{matrixDB}
\begin{array}{lll}
	\left[ \left( dB\right) _{\ell }\right] _{\mathcal{E}_{-t},\mathcal{E}_{t}}
	&=&\left[ \left( dF\right) _{\left( u,t\right) }\right] _{\mathcal{B}_{t},%
		\mathcal{E}_{t}}\left[ dg_{\left( u,-t\right) }\right] _{\mathcal{B}_{-t},%
		\mathcal{B}_{t}}\left( \left[ \left( dF\right) _{\left( u,-t\right) }\right]
	_{\mathcal{B}_{-t},\mathcal{E}_{-t}}\right) ^{-1}   \\
	&& \\
	&=&C_{t}\left( 
	\begin{array}{cc}
		I & 0 \\ 
		0 & R%
	\end{array}%
	\right) \left( C_{-t}\right) ^{-1}=\left( 
	\begin{array}{cc}
		R & 0_{2} \\ 
		D & R%
	\end{array}%
	\right) \text{,}
\end{array}
\end{equation}
where  $C_t$ is as in (\ref{Ct}), $I$ is the $(2\times 2)$-identity matrix, $R=\left( 
\begin{array}{cc}
1 & 0 \\ 
0 & -1%
\end{array}%
\right) $ and 
\begin{equation*}
D=\frac{2}{b}\left( 
\begin{array}{cc}
s_{\kappa }\left( t\right) \kappa b_{22} & \kappa b_{12} \\ 
-b_{21} & -\frac{b_{11}}{s_{\kappa }\left( t\right) }%
\end{array}%
\right)\text{,}
\end{equation*} 
with $b=\det \left(A_{p}\right)$. Thus, 
\begin{equation*}
\begin{array}{c}
(dB)_{\ell }\left( E_{1}^{-t}\right) =E_{1}^{t}+\frac{2}{b}s_{\kappa }\left(
t\right) \kappa b_{22}E_{3}^{t}-\frac{2}{b}b_{21}E_{4}^{t}, \\ 
\\ 
(dB)_{\ell }\left( E_{2}^{-t}\right) =-E_{2}^{t}+\frac{2}{b}\kappa
b_{12}E_{3}^{t}-\frac{2}{b}\frac{b_{11}}{s_{\kappa }\left( t\right) }%
E_{4}^{t},%
\end{array}%
\end{equation*}
\begin{equation*}
(dB)_{\ell }\left( E_{3}^{-t}\right) =E_{3}^{t} \text{\qquad and \qquad} 
 (dB)_{\ell }\left( E_{4}^{-t}\right) =-E_{4}^{t}.
\end{equation*}

Recall from (\ref{omegaKplus}) the definition of the symplectic form $\omega
_{K}$. Straightforward computations yield that 
$$
\left\langle E_{i}^{t}\left(
0\right) \times E_{j}^{t}\left( 0\right) ,u_{t}\right\rangle_{\kappa}
=\left\langle \left( E_{i}^{t}\right) ^{\prime }\left( 0\right) \times
\left( E_{j}^{t}\right) ^{\prime }\left( 0\right)
,u_{t}\right\rangle_{\kappa} =0
$$ 
for all $1\leq i<j\leq 3$, 
except for 
\begin{equation*}
\left\langle E_{2}^{t}\left( 0\right) \times E_{4}^{t}\left( 0\right)
,u_{t}\right\rangle_{\kappa} =\left\langle \left( E_{1}^{t}\right) ^{\prime
}\left(0\right) \times \left( E_{3}^{t}\right) ^{\prime }\left( 0\right)
,u_{t}\right\rangle _{\kappa}=-1\text{.}
\end{equation*}%
Hence, 
\begin{equation*}
\left[ \omega _{K}\right] _{\mathcal{E}_{t}}=\left( 
\begin{array}{cc}
0_{2} & \rho \\ 
-\rho & 0_{2}%
\end{array}%
\right)
\end{equation*}%
with $\rho =\left( 
\begin{array}{cc}
\kappa & 0 \\ 
0 & 1%
\end{array}%
\right) $. We observe that $\left[ \omega _{K}\right] _{\mathcal{E}_{-t}}=%
\left[ \omega _{K}\right] _{\mathcal{E}_{t}}$. Hence, calling $H$ the matrix
in (\ref{matrixDB}), we have to check that 
\begin{equation*}
H^{T}\left[ \omega _{K}\right] _{\mathcal{E}_{t}}H=\left[ \omega _{K}\right]
_{\mathcal{E}_{-t}}
\end{equation*}%
The left hand side equals 
\begin{equation*}
\left( 
\begin{array}{cc}
-D^{T}\rho R+R\rho D & \rho \\ 
-\rho & 0_{2}%
\end{array}%
\right)
\end{equation*}%
and 
\begin{equation*}
-D^{T}\rho R+R\rho D=\frac{2}{\bigskip a}\left( 1-\kappa ^{2}\right) \left( 
\begin{array}{cc}
0 & -b_{12} \\ 
b_{21} & 0%
\end{array}%
\right) ,
\end{equation*}%
which is the zero matrix since $\kappa =\pm 1$, as desired.
\end{proof}

\begin{proof}[Proof of Theorem \protect\ref{thm:symp}(b)]
Following the computations in the proof of part (a), we have 
\begin{equation*}
\left[ \omega _{\times }\right] _{\mathcal{E}_{t}}=\left( 
\begin{array}{cc}
j & 0_{2} \\ 
0_{2} & j%
\end{array}%
\right) ,
\end{equation*}%
where $j=\frac{1}{2}\left( 
\begin{array}{cc}
0 & 1 \\ 
-1 & 0%
\end{array}%
\right) $ and $\left[ \omega _{\times }\right] _{\mathcal{E}_{-t}}=\left[
\omega _{\times }\right] _{\mathcal{E}_{t}}$. Further computations yield 
\begin{equation}
H^{T}\left[ \omega _{\times }\right] _{\mathcal{E}_{t}}H=\left( 
\begin{array}{cc}
j & -(RjD)^{T} \\ 
RjD & j%
\end{array}%
\right) .  \label{HtOmegaH}
\end{equation}%
Now, $RjD=\frac{1}{b}\left( 
\begin{array}{cc}
-b_{21} & -b_{11}/s_{\kappa }\left( t\right) \\ 
s_{\kappa }\left( t\right) \kappa b_{22} & \kappa b_{12}%
\end{array}%
\right) $ and $b_{11}/s_{\kappa }(t)\neq 0$ since by the hypothesis $S$ is
quadratically convex. Therefore, $RjD\neq 0_{2}$ and the expression (\ref%
{HtOmegaH}) is not equal to $\left[ \omega _{\times }\right] _{\mathcal{E}%
_{-t}}$.
\end{proof}

We conclude this section with the following proposition, which relates
Theorem \ref{thm:symp} with plane hyperbolic outer billiards \cite{planeHOB}
and supports the fact that $\omega _{K}$ (in contrast with $\omega _{\times
} $) is the natural symplectic form in our context.

\begin{proposition}
\label{prop:kx} Let $H^{2}$ be a totally geodesic hyperbolic plane in
hyperbolic space.

\smallskip

a) Let $\nu $ be a unit normal vector field on $H^{2}$ and let $%
f:H^{2}\rightarrow \mathcal{G}_{-1}$,\ $f(p)=[\gamma _{\nu (p)}]$. Then $%
f^{\ast }\omega _{K}$ is the area form on $H^{2}$.

\smallskip

b) Let $S$ be a smooth, closed, quadratically convex surface in $H^{3}$
which is invariant by the reflection with respect to $H^{2}$. Then the outer
billiard map on $\mathcal{G}_{-1}$ associated with $S$ preserves the
oriented lines orthogonal to $H^{2}$ and induces in the obvious manner the
outer billiard map on $H^{2}$ determined by the closed strictly convex curve
with image $S\cap H^{2}$. This plane outer billiard preserves the area form
on $H^{2}$.
\end{proposition}

\begin{proof}
For part (a), let $p\in H^{2}$ and let $\left\{ z_{1},z_{2}\right\}$ be a
positively oriented (with respect to the orientation on $H^{2}$ determined
by $\nu$) orthonormal basis of~$T_{p}H^{2}$. For $i=1,2$, let $J_{i}$ be the
Jacobi field along $\gamma _{\nu ( p) }$ satisfying $J_{i}(0) =z_{i}$ and $%
J_{i}^{\prime }( 0) =0$. Using (\ref{omegaKplus}) and that $H^{2}$ is
totally geodesic we have that 
\begin{equation*}
\left( f^{\ast }\omega_{K}\right) _{p}( z_{1},z_{2}) =\omega _{K}(
df_{p}(z_{1}) ,df_{p}( z_{2}) ) =\omega _{K}(J_{1},J_{2}) =1.
\end{equation*}
Part (b) is an immediate consequence of Part (a) and Theorem \ref{thm:symp}.
\end{proof}

\section{Dynamics of the outer billiard map on $\mathcal{G}_{-1}$}\label{Dynamics}

Before proving Proposition \ref{NotParallel}, we comment on the Klein model
of hyperbolic space, that is, the open ball $\mathcal{H}$ centered at the
origin with radius $1$, where the trajectories of geodesics are the
intersections of Euclidean straight lines with the ball. The intersections
of $\mathcal{H}$ with Euclidean planes are totally geodesic hyperbolic
planes.

We recall the following well-known constructions on $\mathcal{H}$ (see for
instance Chapter 6 of \cite{BridsonHaefliger}). For an oriented line $\ell $
in $\mathcal{H}$ we call $\ell ^{+}$ and $\ell ^{-}$ its forward and
backward ideal end points in the two sphere $\partial \mathcal{H}$.

Let $\ell _{1}$ and $\ell _{2}$ be two coplanar oriented lines in $\mathcal{H%
}$ such that the corresponding extensions to Euclidean straight lines
intersect in the complement of the closure of $\mathcal{H}$. In particular, $%
\ell _{1}$ and $\ell _{2}$ do not intersect and are not asymptotic and hence
there exists the shortest segment joining them with respect to the
hyperbolic metric; we call it $s( \ell _{1},\ell _{2}) $.

The hyperbolic midpoint of $s( \ell _{1},\ell _{2}) $ is the intersection of
the Euclidean segments joining $\ell _{1}^{+}$ with $\ell _{2}^{-}$ and $%
\ell _{1}^{-}$ with $\ell _{2}^{+}$, or joining $\ell _{1}^{+} $ with $\ell
_{2}^{+}$ and $\ell _{1}^{-}$ with $\ell _{2}^{-}$, depending on the
orientation of the lines.

Suppose that $\ell _{1}$ and $\ell _{2}$ lie in the plane $P$, let $D = P
\cap \mathcal{H}$, and let $C$ be the boundary of $D$. We describe the
construction of the segment $s( \ell _{1},\ell _{2}) $ in the case when one
of the lines, say $\ell_1$, is a diameter in $D$. Let $p$ be the
intersection of the tangent lines to $C$ through the ideal end points $\ell
_{2}^{+}$ and $\ell _{2}^{-}$. Then $s( \ell _{1},\ell _{2}) $ is the
segment which joins $\ell _{1}$ and $\ell _{2}$ and is contained in the
straight line through $p$ perpendicular to $\ell_{1}$ (the point $p$ is
called the pole of $\ell _{2}$ in the plane $P$).

We recall the formula for the hyperbolic distance between a point in $%
\mathcal{H}$ and the midpoint of any chord containing it: Let $X,Y$ be two
distinct points in $S^{2}=\partial \mathcal{H}$ and let $C$ be the midpoint
of the segment joining $X$ and $Y$, that is, $C=\frac{1}{2}( X+Y) $. Then,
for any $t\in \left( 0,1\right) $,%
\begin{equation}
d( C,C+t\,\tfrac{Y-X}{2}) =\operatorname{arctanh} t,  \label{distanciaKlein}
\end{equation}
where $d$ is the hyperbolic distance in $\mathcal{H}$ (it is not difficult 
to deduce the expression from the second displayed formula of 
Proposition 6.2 in \cite{BridsonHaefliger}).

Since quadratic contact is invariant by diffeomorphisms, a quadratically
convex surface of $\mathbb{R}^{3}$ contained in $\mathcal{H}$ is also
quadratically convex with the hyperbolic metric. By abuse of notation, we
describe an oriented geodesic $\ell $ in $\mathcal{H}$ by the straight
Euclidean line $p+\mathbb{R}u$ containing $\ell $, with $p\in \mathcal{H}$
and $u$ a unit vector giving the orientation.

\begin{proof}[Proof of Proposition \protect\ref{NotParallel}]
We use the Klein model of hyperbolic space. Let $\ell =\mathbb{R}e_{3}$ and $%
\ell _{\theta }=\mathbb{R}( \sin \theta ~ e_2 + \cos \theta ~ e_3)$. We
construct a surface $S$ contained in the region $x\geq 0,y\geq 0$ of $%
\mathcal{H}$ whose associated outer billiard map $B$ satisfies $B^{3}( \ell
) =\ell _{\theta }$. We fix a real number $r$ in the interval $\left( \sin
\theta ,1\right) $. Let $\ell _{1}=re_{1}+\mathbb{R}e_{3}$ and $\ell
_{2}=re_{2}+\mathbb{R}e_{3}$, and let $p=\left( x_0,0,0\right) $ and $%
q=\left( 0,y_0,z_0\right) $ be the hyperbolic midpoints between $\ell $ and $%
\ell _{1}$ and between $\ell _{\theta }$ and $\ell _{2}$, respectively (see
Figure \ref{fig:Klein}).

\begin{figure}[ht!]
\centerline{
\includegraphics[totalheight=1.8in,angle=0]{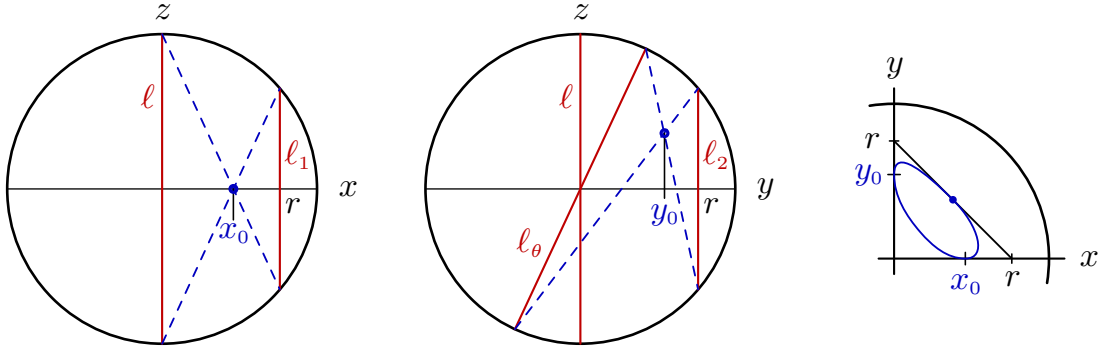}
}
\caption{Elements for the construction of $S$}
\label{fig:Klein}
\end{figure}

There exists a smooth, closed, quadratically convex surface $S$ contained in 
$\mathcal{H}$ and tangent to the vertical planes $y=0$,\ $y+x=r$ and $x=0$,
at the points $p,$ $\left( r/2,r/2,0\right) $ and $q$, respectively. In
fact, consider a quadratically convex compact surface $S^{\prime }$ tangent
to those planes at the points $p,$ $\left( r/2,r/2,0\right) $ and $\left(
0,y_{0},0\right) $, respectively, such that the absolute value of the height
function $\left. z\right\vert _{S^{\prime }}$ is bounded by $\varepsilon $
for some $\varepsilon >0$. Let $T$ be the unique affine transformation of $%
\mathbb{R}^{3}$ fixing the vertical plane through $p$ and $\left(
r/2,r/2,0\right) $ and sending $\left( 0,y_{0},0\right) $ to $q$. Then $%
S=T\left( S^{\prime }\right) $ satisfies the desired conditions (in particular, it
preserves the vertical planes and the quadratical contact), provided
that $\varepsilon $ is small enough.
By the properties of $S$ we have that $B(
\ell ) =\ell _{1}$, $B^{2}( \ell ) =\ell _{2}$ and $B^{3}( \ell ) =\ell
_{\theta }$.
\end{proof}

%\smallskip

\begin{proof}[Proof of Proposition \protect\ref{holonomy}]
As in the proof of Proposition \ref{NotParallel}, we use the Klein model $%
\mathcal{H}$ for hyperbolic space. We write $\mathbb{R}^{3}=\mathbb{C}\times 
\mathbb{R}$. Given $0<a<\frac{1}{2}<$ $r_{o}<1$ and $h_{o}=\sqrt{1-r_{o}^{2}}
$, we consider the straight lines%
\begin{equation*}
\begin{array}{ll}
\gamma _{0}( t) =( t,0), & \gamma _{1}( t) =( \tfrac{i}{2}+t( 1-ai) ,0), \\ 
\gamma _{2}( t) =( \tfrac{i}{2}+t( 1-ai) ,h_{o}), & \gamma _{3}( t) =(
t,h_{o}) .%
\end{array}%
\end{equation*}

For $k=0,1,2,3$, let $\ell _{k}$ be the corresponding oriented geodesic in $%
\mathcal{H}$ and set $\ell _{4}=\ell _{0}$. Notice that $\ell_k$ and $\ell_{k+1}$ 
are coplanar for any $k=0,1,2,3$. We
call $P_{k}$ the hyperbolic plane containing $\ell _{k}$ and $\ell _{k+1}$.
Since $0<2a<1$, the lines $\ell_k$ and $\ell_{k+1}$ are disjoint and not asymptotic, and
so the shortest segment $\sigma _{k}$ joining them is well-defined.

We will show the existence of a smooth, closed, quadratically convex surface 
$S$ in $\mathcal{H}$ not intersecting $\ell _{0}$ such that the associated
billiard map $B$ satisfies $B^{k}( \ell _{0}) =\ell _{k}$ for $k=0,\dots ,4$
and its holonomy at $\ell _{0}$ is not trivial.

We make computations for general values of $r$ and $h=\sqrt{1-r^{2}}$, with $%
0<a<\frac{1}{2}<r\leq 1$, in order to deal simultaneously with $\ell _{0}$
and $\ell _{1}$ on the one hand (case $r=1$) and $\ell _{2}$ and $\ell _{3}$
on the other (case $r=r_{o}$), since the former lie in a disc of radius $1$
at height $0$ and the latter in a disc of radius $r_{o}$ at height $h_{o}$.

The end points of $\ell _{1}$ and $\ell _{2}$ are given by%
\begin{equation*}
\ell _{1}^{\varepsilon }=\left( z_{\varepsilon }( 1), 0\right) \text{\ \ \ \
\ \ and\ \ \ \ \ \ }\ell _{2}^{\varepsilon }=\left( z_{\varepsilon }(
r_{o}), h_{o}\right) \text{,}
\end{equation*}%
for $\varepsilon =\pm 1$, where $z_{\varepsilon }( r) =\frac{i}{2}%
+t_{\varepsilon }( r) \left( 1-ai\right) $, with $t_{-}( r) <t_{+}( r) $
being the solutions of the equation $t^{2}+\left( \frac{1}{2}-at\right)
^{2}=r^{2}$. Since $\ell _{1}$ and $\ell _{2}$ are parallel, the end points
of $\sigma _{1}$ are the respective midpoints, whose (common) component in $%
\mathbb{C}$ is 
\begin{equation*}
z_{o}=\tfrac{1}{2}\left( z_{+}( 1) +z_{-}( 1) \right) =\tfrac{1}{2}\left(
z_{+}( r_{o}) +z_{-}( r_{o}) \right) =\tfrac{1}{2\left( a^{2}+1\right) }%
\left( a+i\right) \text{.}
\end{equation*}%
Hence, $q_{+}( \ell _{1}) =\left( z_{o},0\right) $ and\ $q_{-}( \ell _{2})
=\left( z_{o},h_{o}\right) $. Similarly, $q_{-}( \ell _{0}) =\left(
0,0\right) $ and $q_{+}( \ell _{3}) =\left( 0,h_{o}\right) $.

\medskip

\begin{figure}[ht!]
\centerline{
\includegraphics[width=3.3in]
{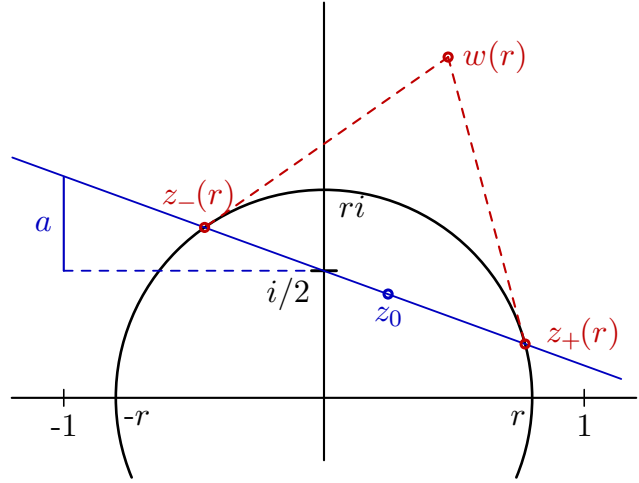}
}
\caption{Elements for the construction of $S$}
\label{fig:holonomy}
\end{figure}

Let $p_{1}=\left( w( 1), 0\right) $ and $p_{2}=\left( w( r_{o}),
h_{o}\right) $ be the poles of the line $\ell _{1}$ in the plane $\mathbb{C}%
\times \left\{ 0\right\} $ and of the line $\ell _{2}$ in the plane $\mathbb{%
C}\times \left\{ h_{o}\right\} $, respectively. That is, $w( r) $ is the
intersection of the lines tangent to the circle of radius $r$ in $\mathbb{C}$
at the points $z_{-}( r) $ and $z_{+}( r) $, in the horizontal plane at height $h$, 
for $h=0,h_o$  
 (see Figure \ref{fig:holonomy}).
Using the construction of the shortest segment joining two oriented lines,
the segment $\sigma _{2}$ is contained in the line perpendicular to $\ell
_{3}$ passing through $\left( w\left( r_{o}\right) ,h_{o}\right) $, that is,
the line $\left( \operatorname{Re}w\left( r_{o}\right) +\mathbb{R}i,h_{o}\right) $.
Putting $r=1$, we get that $\sigma _{0}$ is contained in the line $\left( 
\operatorname{Re}w\left( 1\right) +\mathbb{R}i,0\right) $.

We have that $w( r) =z_{+}( r) +s_{o}( r) iz_{+}( r) $, where $s_{o}$ is the
solution of the equation%
\begin{equation*}
z_{+}( r) +siz_{+}( r) =z_{-}( r) -siz_{-}( r).
\end{equation*}%
A straightforward computation yields $\operatorname{Re}w( r) =2ar^{2}$.

Now, computing the intersections of the remaining $\sigma _{k}$ with the
lines $\ell _{j}$, we obtain the rest of the $q_{\pm }( \ell _{j}) $: 
\begin{equation*}
\begin{array}{ll}
q_{+}( \ell _{0}) =\left( \operatorname{Re}w( 1) ,0\right) =2\left( a,0\right) \text{%
,} & q_{-}( \ell _{1}) =2\left( a+i\left( \frac{1}{4}-a^{2}\right) ,0\right)
\\ 
q_{+}( \ell _{2}) =\left( 2\left( ar_{o}+i\left( \frac{1}{4}%
-a^{2}r_{o}^{2}\right) \right) ,h_{o}\right) \text{,} & q_{-}( \ell _{3})
=\left( \operatorname{Re}w( r_{o}) ,h_{o}\right) =\left( 2ar_{o}^{2},h_{o}\right) 
\text{.}%
\end{array}%
\end{equation*}

As in Proposition \ref{NotParallel}, for each $a>0$ there exists a
quadratically convex surface $S_{a}$ in $\mathcal{H}$ tangent to the plane $%
P_{k}$ at the midpoint of $\sigma _{k}\,$for any $k=0,\dots ,3$. The
associated billiard map $B_{a}$ satisfies $\left( B_{a}\right) ^{4}( \ell
_{0}) =\ell _{0}$ and its holonomy at $\ell _{0}$ turns out to be 
\begin{equation*}
H( a) =\sum\nolimits_{k=0}^{3}\left( -1\right) ^{k}d( q_{+}( \ell _{k})
,q_{-}( \ell _{k}) ) .
\end{equation*}

Particularizing $r_{o}=h_{o}=1/\sqrt{2}$, using (\ref{distanciaKlein}) we
obtain that%
\begin{equation*}
H( a) =\text{arctanh}\left( 2a\right) -\text{arctanh}\left( a\sqrt{4a^{2}+3}%
\right) +\text{arctanh}\left( a\sqrt{2a^{2}+1}\right) -\text{arctanh}\left(
a\right) \text{.}
\end{equation*}

We compute $H( 0) =0$ and $H^{\prime }( 0) \neq 0$. Therefore for
sufficiently small $a>0$, the holonomy of $B_{a}$ at $\ell _{0}$ does not
vanish.
\end{proof}

\vspace{0.1cm}

\noindent Yamile Godoy\newline
\noindent FAMAF (Universidad Nacional de C\'ordoba) and CIEM (Conicet) \newline
\noindent Ciudad Universitaria, X5000HUA C\'{o}rdoba, Argentina\newline
\noindent yamile.godoy@unc.edu.ar

\vspace{0.5cm}

\noindent Michael Harrison\newline
\noindent Institute for Advanced Study\newline
\noindent 1 Einstein Drive\newline
\noindent Princeton, NJ 08540, US\newline
\noindent mah5044@gmail.com

\vspace{0.5cm}

\noindent Marcos Salvai\newline
\noindent FAMAF (Universidad Nacional de C\'ordoba) and CIEM (Conicet) \newline
\noindent Ciudad Universitaria, X5000HUA C\'{o}rdoba, Argentina \newline
\noindent marcos.salvai@unc.edu.ar


\begin{thebibliography}{99}
\bibitem{Alek} D. V. Alekseevsky, B. Guilfoyle, W. Klingenberg, \textsl{On
the geometry of spaces of oriented geodesics}, Ann. Glob. Anal. Geom. 40
(2011) 389--409.

\bibitem{Anciaux} H. Anciaux, \textsl{Spaces of geodesics of
pseudo-Riemannian space forms and normal congruences of hypersurfaces},
Trans. Amer. Math. Soc. 366 (2014), 2699--2718

\bibitem{besse} A. Besse, \textsl{Manifolds all of whose geodesics are closed%
}. Ergebnisse der Mathematik und ihre Grenzgebiete 93. Springer,
Berlin\,-\,New York (1978).

\bibitem{BridsonHaefliger} M. R. Bridson, A. Haefliger, \textsl{Metric
spaces of non-positive curvature}, Grundlehren der mathematischen
Wissenschaften 319. Springer, Berlin\,-\,New York (1999).

\bibitem{docarmo} M. do Carmo. \emph{Riemannian geometry}, Springer (1992).

\bibitem{doCarmoWarner} M. P. do Carmo, F. W. Warner, \textsl{Rigidity and
convexity of hypersurfaces in spheres}, J. Diff. Geom. 4 (1970) 133--144.

\bibitem{TabachMathIntell} F. Dogru, S. Tabachnikov, \textsl{Dual billiards}%
, Math. Intelligencer 27 (2005) 18--25.

\bibitem{GG} N. Georgiou, B. Guilfoyle, \textsl{On the space of oriented
geodesics of hyperbolic 3-space}, Rocky Mountain J. Math. 40 (2010)
1183--1219.

\bibitem{brendan} N. Georgiou, B. Guilfoyle, \textsl{Hopf hypersurfaces in
spaces of oriented geodesics}, J. Geom. 108 (2017) 1129--1135.

\bibitem{GW} H. Gluck, F. Warner, \textsl{Great circle fibrations of the
three-sphere}, Duke Math. J. 50 (1983) 107--132.

\bibitem{GodoySalvaiMag} Y. Godoy, M. Salvai, \textsl{The magnetic flow on
the manifold of oriented geodesics of a three dimensional space form}, Osaka
J. Math. 50 (2013) 749--763.

\bibitem{GShip} Y. Godoy, M. Salvai, \textsl{Global geodesic foliations of
the hyperbolic space}, Math. Z. 281 (2015) 43--54.

\bibitem{GodoySalvaiCali} Y. Godoy, M. Salvai, \textsl{Calibrated geodesic
foliations of hyperbolic space}, Proc. Amer. Math. Soc. 144 (2016) 359--367.

\bibitem{gk1} B. Guilfoyle and W. Klingenberg, \textsl{An indefinite
K\"ahler metric on the space of oriented lines}, J. London Math. Soc. 72
(2005) 497--509.
 
\bibitem{HarrisonMZ} M. Harrison. \textsl{Skew flat fibrations}, Math. Z. 282
(2016) 203--221.

\bibitem{HarrisonBLMS} M. Harrison. \textsl{Contact structures induced by skew 
fibrations of $\mathbb{R}^3$}, Bull. Lond. Math. Soc. 51 (2019) 
887--899.

\bibitem{HarrisonAGT} M. Harrison. \textsl{Fibrations of $\mathbb{R}^{3}$ by
oriented lines}, Algebr. Geom. Topol. 21 (2021) 2899--2928.

\bibitem{HarrisonT} M. Harrison. \textsl{Skew and sphere fibrations}, to appear in \emph{Trans. Amer. Math. Soc.}, arxiv:2203.16412 [math.GT].

\bibitem{Hitchin} N. J. Hitchin, \textsl{Monopoles and geodesics}, Commun.
Math. Phys. 83 (1982) 579--602.

\bibitem{FM} F. Morgan, \textsl{The exterior algebra $\Lambda^k{\mathbb{R}}%
^n $ and area minimization}, Linear Algebra Appl. 66 (1985) 1--28.

\bibitem{moser1} J. Moser, \textsl{Stable and random motions in dynamical
systems}. Ann. of Math. Stud. 77, Princeton (1973).

\bibitem{moser2} J. Moser, \textsl{Is the solar system stable?}, Math.
Intelligencer 1 (1978) 65--71.

\bibitem{salvaimm} M. Salvai, \textsl{On the geometry of the space of
oriented lines of Euclidean space}, Manuscripta Math. 118 (2005) 181--189.

\bibitem{Salvai2} M. Salvai, \textsl{On the geometry of the space of
oriented lines of the hyperbolic space}, Glasgow Math. J. 49 (2007) 357--366.

\bibitem{salvaiolf} M. Salvai, \textsl{Global fibrations of }$\mathbb{R}^{3}$
\textsl{by oriented lines}, Bull. London Math. Soc. 41 (2009) 155--163.

\bibitem{TabachnikovOuter} S. Tabachnikov, \textsl{Outer billiards}, Russian
Math. Surveys 48 (1993) 81--109.

\bibitem{TabachnikovBilliards} S. Tabachnikov, \textsl{Billiards}. Panoramas
et Synth\`{e}ses 1, Soci\'{e}t\'{e} Math\'{e}matique de France (1995).

\bibitem{Tabachnikov} S. Tabachnikov, \textsl{On the dual billiard problem},
Adv. Math. 115 (1995) 221--249.

\bibitem{TabachnikovPeriodic} S. Tabachnikov, \textsl{On three-periodic
trajectories of multi-dimensional dual billiards}, Algebr. Geom. Topol. 3
(2003) 993--1004.

\bibitem{planeHOB} S. Tabachnikov, \textsl{Dual billiards in the hyperbolic
plane}, Nonlinearity 15 (2002) 1051--1072.

\bibitem{bookT} S. Tabachnikov, \textsl{Geometry and billiards}. Student
Mathematical Library 30. American Mathematical Society, Providence RI (2005).
\end{thebibliography}
\end{document}